\documentclass{amsart}
\usepackage{amsmath, amsfonts, amssymb,amsthm}
\usepackage{amstext}
\usepackage{mathrsfs}
\usepackage[dvipsnames]{xcolor}
\usepackage{tikz}
\usepackage[all]{xy}
\usepackage{enumerate}
\usetikzlibrary{er,positioning}

\usepackage{lineno}
\usepackage{comment}

\theoremstyle{definition}
\newtheorem{definition}{Definition}[section]

\theoremstyle{remark}

\theoremstyle{plain}
\newtheorem{theorem}[definition]{Theorem}
\newtheorem{lemma}[definition]{Lemma}
\newtheorem{proposition}[definition]{Proposition}
\newtheorem{corollary}[definition]{Corollary}

\newcommand{\sq}[1]{\ifx#1([\else\ifx#1)]%
  \else\message{invalid use of "sq"}\fi\fi}

\newcommand{\PP}{\mathbb{P}}

\newcommand{\leqexc}{\leq_{\operatorname{exc}}}

\DeclareMathOperator{\ord}{ord}

\DeclareMathOperator{\Supp}{supp}

\DeclareMathSymbol{\idot}{\mathbin}{operators}{`\.}
\allowdisplaybreaks
\begin{document}
\title[Nevanlinna pair and algebraic hyperbolicity]{Nevanlinna pair and algebraic hyperbolicity}
\author{Yan He}
\address{Department of Mathematics\newline
\indent University of Houston\newline
\indent Houston,  TX 77204, U.S.A.} 
\email{yanhe@math.uh.edu}
\author{Min Ru}
\address{Department of Mathematics\newline
\indent University of Houston\newline
\indent Houston,  TX 77204, U.S.A.} 
\email{minru@math.uh.edu}{\bf }

\begin{abstract} We introduce the notion of the {\it
Nevanlinna pair} for a pair $(X, D)$, where $X$ is a projective variety and $D$ is an effective Cartier divisor on $X$.
 This notion links  and unifies the Nevanlinna theory, the complex hyperbolicity (Brody and Kobayashi hyperbolicity), 
the big Picard type extension theorem (more generally the Borel hyperbolicity), as well as the algebraic hyperbolicity. The key is to use the Nevanlinna theory on parabolic Riemann surfaces recently developed by
P${\rm{\breve{a}}}$un and Sibony \cite{paun2014value}.
\newline
\begin{center}
{\it Dedicated to Professor Shiing-Shen Chern on his 110th birth anniversary} 
\end{center}
\end{abstract}
 \thanks{2010\ {\it Mathematics Subject Classification.}
32H30,  32Q45, 32A22, 53C60.}  
\thanks{The second named author is supported in part by Simon Foundations grant award \#531604.}

\baselineskip=16truept \maketitle \pagestyle{myheadings}
\markboth{}{Nevanlinna pair and algebraic hyperbolicity}
\section{Introduction}
The classical big Picard theorem says that every holomorphic map
from the punctured unit disk $\bigtriangleup^*$ into ${\mathbb P}^1({\mathbb C})$
 whose image omits three points can be extended to a
holomorphic map  from $\bigtriangleup$ into ${\mathbb P}^1({\mathbb C})$. On the other hand, it is well known that $\mathbb{P}^1(\mathbb{C})$ minus three points is hyperbolic.
This suggests  that the big-Picard-type results are strongly related to the complex hyperbolicity. In fact,
a theorem of Kwack and Kobayashi states that every holomorphic map $f:  \bigtriangleup^*
\rightarrow U$ extends to a holomorphic map $f:  \bigtriangleup
\rightarrow \overline{U}$  if $U$  is a quasi-projective variety and is hyperbolically embedded in some compactification $\overline{U}$ (cf. \cite{kobayashi2013hyperbolic}, Theorem 6.3.7 or \cite{lang1987introduction}, II $\S 2$), a notion invented by Kobayashi expressly
for this purpose.

 The big-Picard-type extension problem is important and is often  linked to the study of the algebraicity of holomorphic maps into a fixed variety. It recently attracted our interest because of the work  by Javanpeykar-Kucharczyk \cite{Javanpeykar2020} on the algebraicity of analytic
 maps. According to \cite{Javanpeykar2020}, let $X$ and $Y$ be two  finite
 type schemes over ${\mathbb C}$, their associated analytic spaces are denoted by $X^{an}$, $Y^{an}$. Let $\phi: X^{an} \rightarrow Y^{an}$
 be a holomorphic map. The 
 map  $\phi$  is said to be {\it algebraic} if there is a
 morphism of ${\mathbb C}$-schemes $f: X\rightarrow Y$ such that $f^{an}
  = \phi$. To ``algebraize" (in the language of \cite{Javanpeykar2020}),  one 
 usually needs the extension property. 
 For this reason, in  \cite{Javanpeykar2020}, they introduced a new notion of hyperbolicity: a quasi-projective variety $U$ is said to be {\it Borel hyperbolic} if any holomorphic map from a quasi-projective variety $V$ to $U$ is necessarily algebraic.
 They proved that if the big Picard property for a 
 quasi-projective variety $U$ holds, then $U$ is Borel hyperbolic. We refer the reader to  \cite{Javanpeykar2020},  $\S1$,  for their
 motivation on the Borel hyperbolicity.  Subsequently, Deng-Lu-Sun-Zuo \cite{deng2020picard}, and Deng \cite{deng2020bbig} also obtained results about the big Picard property.
  For the same  reason, the big-Picard-type results
are also  related to the algebraic hyperbolicity, a notion introduced 
 by J. P. Demailly \cite{Dem} and generalized to the logarithmic pairs by Xi Chen \cite{chen}. It is known, due to the result of Pacienza and Rousseau \cite{pacienza2007logarithmic}, 
 if $X\setminus D$ is hyperbolically imbedded in $X$ then $(X, D)$ is algebraically  hyperbolic.
 
 The original proof of Kwack and Kobayashi's extension theorem is rather conceptual and involved. It seems to us that the
Nevanlinna theory (in particular, the application of the logarithmic derivative lemma) is a more natural approach to attack the Picard extension problems, as shown by M. Green \cite{Green1975}. Recently, Y.T. Siu \cite{siu2015hyperbolicity} revived this method. 
To study the  $\bigtriangleup^*$-extension property, we make a  change of  variable $z:=1/\zeta$.
Then the problem is reduced to the existence of essential singularity at $\infty$. It is known that,  
for a holomorphic map $\phi: {\mathbb C}- \overline{\bigtriangleup(r_0)}\rightarrow {\mathbb P}^n({\mathbb C})$, 
if $T_{\phi}(r, r_0)\leqexc O(\log r)$ (see below for the notation) then $\phi$  can be extended to a holomorphic
map from ${\mathbb C}\cup\{\infty\}-\overline{\bigtriangleup(r_0)}$ into ${\mathbb P}^n({\mathbb C})$.  For this reason, we 
introduce the concept of {\it quantitatively big Picard} (see Definition \ref{quantitativep} below), and prove the following result. 
 \begin{theorem}[See Theorem \ref{qp2}] If $X\setminus D$ is hyperbolically imbedded in $X$, then  $X\setminus D$ is quantitatively big Picard.
 \end{theorem}

 More generally, we introduce the concept of {\it Nevanlinna pair} through holomorphic maps on an open  parabolic Riemann surface. 
  We briefly recall some notations (For details, see Section 2 below and \cite{paun2014value}).  A
  non-compact Riemann surface $Y$ is {\it parabolic} if it  admits a parabolic exhaustion
  function, i.e. a continuous proper function $\sigma: Y\rightarrow [0, \infty)$ such that $\log \sigma$ is harmonic outside a
  compact subset of $Y$.  However, in this paper, we restrict it to a special case for simplicity which is sufficient for our purpose, 
 i.e.,  by a {\it parabolic Riemann surface} we mean an open Riemann surface $Y$, together with a proper continuous function $\sigma:Y\to [0,\infty)$ (called a parabolic exhaustion function) such that
  \begin{itemize}
	\item $\log \sigma$ is harmonic outside possibly a finite set  $\Sigma:=\{P_1, \dots, P_k\}$ on $Y$.
	\item At each $P_i\in \Sigma$, in a coordinate chart $(U,z)$ centered at $P_i$ that does not contain other points in $\Sigma$, $\log \sigma(z)   =k_i \log |z| +h_{P_i}(z)$, where $h_{P_i}$ is a harmonic function on $U$.
  \end{itemize}
 Note this notion of parabolic Riemann surface, first appeared in \cite{Gasbarri}, is slightly stronger than the standard one given in \cite{paun2014value}.
We also note that all the inequalities we established in the paper still hold on a standard parabolic Riemann surface (in the sense of \cite{paun2014value}) with the term $2\varsigma\log r$ (see below) replaced by a constant multiple $O(\log r)$ of $\log r$ depending on $Y$.
Denote by $B(r)$ the parabolic disk $\{y\in Y:\sigma(y) <r\}$ and by $S(r)$ the parabolic circle $\{y\in Y:\sigma(y)=r\}$. 
By Sard's theorem,  $S(r)$ is smooth for almost all $r>0$, in that case we denote by $\mu_r$ the measure induced by the differential $d^c\log\sigma|_{S(r)}$ and write $d\mu_r = d^c\log\sigma|_{S(r)}$. 
Let 
\begin{equation}\label{eqn:defsigma}
\varsigma:=\int_{S(r)} d\mu_r,
\end{equation}
which is independent of $r$ for $r$ large enough because $\log \sigma$ is harmonic. Let 
 $\chi_{\sigma}(r)$ be the Euler characteristic of $B(r)$, and define 
 \begin{equation}\label{wEuler}
 \mathfrak{X}_{\sigma}(r):=\int_{1}^r \chi_{\sigma}(t)\frac{dt}{t}.
 \end{equation}

 Throughout the paper, we fix a nowhere vanishing global holomorphic vector field $\xi\in \Gamma(Y, T_Y)$  on $Y$. Such a vector field exists because $Y$ is non-compact and consequently its holomorphic tangent bundle $T_Y$ is holomorphically trivial (see Theorem 5.3.1, \cite{forstnerivc2011stein}). We also define 
 \begin{equation}\label{wEulernew}
 \mathfrak{E}_{\sigma}(r):=\int_{S(r)}\log^- |d\sigma(\xi)|^2 d\mu_r,
 \end{equation}
 where, for  a positive real number $x$, $\log^+ x= \max\{0, \log x\}$ and $\log ^-x=-  \min\{0, \log x\}$. 
Note that $ \mathfrak{E}_{\sigma}(r)$ is closely related to  $\mathfrak{X}_{\sigma}(r)$.  When $\mathfrak{X}_{\sigma}(r)=O(\log r)$, we also have $\mathfrak{E}_{\sigma}(r)=O(\log r)$. More precisely,  
 according to Lemma \ref{bnew}, we have
 $$\int_{S(r)} \log |d\sigma(\xi)|^2 d\mu_r = -\mathfrak{X}_{\sigma}(r)+ 2\varsigma\log r+O(1).$$

 	\begin{definition}\label{nevp}
	Let $X$ be a projective variety and $D$ be an effective Cartier divisor on $X$.  
	We say that $(X, D)$ is a {\it Nevanlinna pair} if there is a positive $(1, 1)$-form $\eta$ on $X$ such that for any parabolic Riemann surface $Y$ and every holomorphic map $f: Y\rightarrow X$ with $f(Y)\not\subset D$ and for $\delta>0$, one has  $$T_{f, \eta}(r)\leqexc \overline{N}_f(r, D) - \mathfrak{X}_{\sigma}(r) + (\delta+2\varsigma)\log r+ \mathfrak{E}_{\sigma}(r)+ O(1),$$ 
	where $O(1)$ is a constant which may depend on $f$ and $Y$, $\varsigma$ is the constant given by (\ref{eqn:defsigma}), and  $\leqexc$ means that the inequality holds for all $r\in (0, \infty)$ except a set of finite measure depending on $\delta$.	
	 	 \end{definition}

\noindent{\bf Remark}:  Note that the coefficient $(\delta+2\varsigma)$ appeared before $\log r$ is crucial in dealing with 
the algebraic hyperbolicity, hence one of the main focuses is to keep the constants appeared in the paper 
to be independent of $f$ and $Y$. For this reason, throughout the paper, when we mention a constant $C>0$, we always mean that $C$ is independent of $f$ and $Y$ unless otherwise specified. 
On the other hand, when we write $O(\log r)$ or $O(1)$, we mean that the involved constants may depend on $f$ and $Y$.

Note that the complex plane $\mathbb{C}$ together with exhaustion function $\sigma(z)=|z|$ is a parabolic Riemann surface with $\varsigma = \frac{1}{2}$. In this case, $\mathfrak{X}_\sigma(r) = \log r, \mathfrak{E}_\sigma(r) = O(1)$, so if   $(X,D)$ is a Nevanlinna pair then there exists a positive $(1,1)$-form $\eta$ on $X$ such that for every holomorphic map $f:\mathbb{C}\to X\setminus D$ and $\delta > 0$,
$$T_{f,\eta}(r)\leqexc\delta \log r+O(1).$$
This implies that $f$ is constant. Hence we have the following statement: {\it If $(X,D)$ is a Nevanlinna pair, then $X\setminus D$ is Brody hyperbolic}.

Using the recent result of Brotbek and Brunebarbe (\cite{BB} Theorem 6.2), we prove the following result. 
\begin{theorem}[See Theorem \ref{bb}]  \label{1.3}{If $X\setminus D$ is hyperbolically imbedded in $X$, then  $(X, D)$ is a Nevanlinna pair.}
\end{theorem}
 According to Demailly \cite{Dem} and Chen \cite{chen},  the pair $(X,D)$ is said to be  {\it algebraically hyperbolic}
if there exists a positive $(1,1)$-form $\omega$ on $X$ such that 
for any compact Riemann surface $R$ and every holomorphic map 
$f: R\rightarrow X$ with $f(R)\not\subset D$, the following inequality holds
$$\int_R f^*\omega \leq \bar{n}_f(D)+\max\{0, 2g-2\},$$
 where $\bar{n}_f(D)$ is the number of points of $f^{-1}(D)$ on $R$ and $g$ is the genus of $R$.
We  show that Nevanlinna pair implies algebraically hyperbolic.
\begin{theorem}[See Theorem \ref{alg}] If $(X, D)$ is a Nevanlinna pair, then 
	$(X, D)$ is algebraically hyperbolic.
\end{theorem}
We summarize our results in the following diagram:

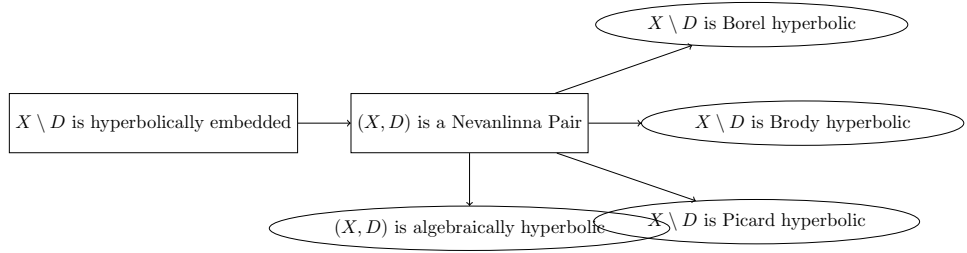
\scalebox{.70}{
	\begin{tikzpicture}
	\node[entity](embedded){$X\setminus D$ is hyperbolically embedded};
	\node[entity](pair)[right=of embedded]{$(X,D)$ is a Nevanlinna Pair} edge[<-] (embedded);
	\node[attribute](borel)[above right=of pair]{$X\setminus D$ is Borel hyperbolic} edge[<-] (pair);
	\node[attribute](picard)[below right =of pair]{$X\setminus D$ is Picard hyperbolic} edge[<-](pair);
	\node[attribute](brody)[right =of pair]{$X\setminus D$ is Brody hyperbolic} edge[<-](pair);
	\node[attribute](algebraic)[below = of pair]{$(X,D)$ is algebraically hyperbolic} edge[<-](pair);
	% Now place a relation (ID=rel1)
	\end{tikzpicture}}

\bigskip We  conjecture that if $(X, D)$ is a Nevanlinna pair, then $X\setminus D$ is Kobayahsi hyperbolic. Therefore, in our opinion,  {\it  Nevanlinna pair} is a suitable notion to unify the Nevanlinna theory, the complex hyperbolicity (Brody and Kobayashi hyperbolicity),  the big Picard type extension property (more generally the Borel hyperbolicity), as well as the algebraic hyperbolicity. 

When $X$ is ${\mathbb P}^n({\mathbb C})$ and $D$ consists of hyperplanes, we prove that 
 Nevanlinna pair,  Brody hyperbolicity and the big Picard type extension theorem (more generally the Borel hyperbolicity) are indeed equivalent. 
 \begin{theorem}[See Theorem \ref{hyper1}] 
 	 Let ${\mathcal H}$ be a finite set of hyperplanes in ${\mathbb P}^n({\mathbb C})$.  Let $|{\mathcal H}|:=\sum_{H\in \mathcal{H}}H$. 
	Then the following statements are equivalent.
	\begin{enumerate}[(a)]
		\item $({\mathbb P}^n(\mathbb{C}), |{\mathcal H}|)$ is a Nevanlinna pair.
		\item ${\mathbb P}^n({\mathbb C})\backslash |{\mathcal H}|$ is Brody hyperbolic.
		\item ${\mathbb P}^n({\mathbb C})\backslash |{\mathcal H}|$ is Picard hyperbolic.
	\end{enumerate}
\end{theorem}
In this paper, we also provide examples  of the Nevanlinna pair $(X, D)$  where, in some cases,
$X\setminus D$ may not be hyperbolically imbedded in $X$. Indeed, in each case, we obtain a more precise Second Main Theorem 
type inequality for holomorphic maps on open parabolic Riemann surfaces.

\begin{theorem}[See Theorem \ref{thm1} and  Theorem \ref{alg}]  Let $X$ be a projective manifold of dimension $n\ge 2$ and let $A$ be
	a very ample line bundle over $X$. Let $D \in |A^m|$ be a general smooth hypersurface with
	$$m\ge (n + 2)^{n+3}(n + 1)^{n+3}.$$ Then there exists a constant $C>0$ such that  for every holomorphic map $f: Y\rightarrow X$ with $f(Y)\not\subset D$  where $Y$ is a parabolic Riemann surface, we have, for $\delta>0$, 
	$$T_{f, A}(r) \leqexc {\overline N}_f(r,D) + C(\log T_{f, A}(r) - \mathfrak{X}_{\sigma}(r)+  (\delta+2\varsigma)\log r+\mathfrak{E}_{\sigma}(r)) +O(1).$$ 
In particular $(X, D)$ is a Nevanlinna pair.   The pair $(X,D)$ is also algebraically hyperbolic.	 
\end{theorem}

\begin{theorem}[See Theorem \ref{abelian} and  Theorem \ref{alg}]\label{thm:intro-abelian} Let $A$ be an abelian variety and 
	$D$ be an ample divisor on $A$. Then there exists a constant $C>0$ such that    for every holomorphic map $f: Y\rightarrow A$ 
	with $f(Y)\not\subset D$ where $Y$ is a parabolic Riemann surface, we have, for
some $k_0>0$, 
	$$T_{f, D}(r)\leqexc {\overline N}^{[k_0]}_f(r, D)+C(\log T_{f, A}(r) - \mathfrak{X}_{\sigma}(r)+  (\delta+2\varsigma)\log r+ \mathfrak{E}_{\sigma}(r))+ O(1).$$
	In particular, 
	$(A, D)$ is 
	a Nevanlinna pair. The pair $(A,D)$ is also algebraically hyperbolic.
	\end{theorem}
 We note that recently Yamanoi \cite{yamanoi} proved that $A\setminus D$ is Kobayashi hyperbolic  under the conditions in Theorem \ref{thm:intro-abelian}.
 
\noindent{\bf Remark}. During the time we were working on the manuscript, we were not aware whether {\it $X\setminus D$ is 
 hyperbolically imbedded in $X$} implies that {\it
$(X, D)$ is a Nevanlinna pair}, so we put this as a conjecture in the original version. When the manuscript was nearly finished, we discovered the preprint \cite{BB} and found that Theorem 6.2 in \cite{BB} is 
exactly, thanks to Brotbek and Brunebarbe,  what we are looking for. Therefore, we decided to include their result as Theorem \ref{1.3} (as well as Theorem \ref{bb}) in the manuscript. 
\section{Quantitative big Picard and the Nevanlinna pair}

\subsection{Relating different analytic notions of hyperbolicity
(Picard hyperbolicity in \cite{deng2020bbig}  and Borel hyperbolicity in \cite{Javanpeykar2020})}

 We first gather known results relating
the different notions of hyperbolicity. We start with an extension property for holomorphic maps. 
Denote by $\bigtriangleup(r)$  the disk centered at the origin with radius $r$, and by $\bigtriangleup^*(r): =\bigtriangleup(r)- \{0\}$ the puncture disk.
We use $\bigtriangleup$ for the unit disk and $\bigtriangleup^*$ the unit punctured disk.

\begin{definition}[\cite{Javanpeykar2020}, \cite{deng2020bbig}]
A finite type separated scheme $X$ over ${\mathbb C}$ has the {\it $\bigtriangleup^*$-extension property} (see \cite{Javanpeykar2020}),  or is
	{\it	Picard hyperbolic} (see \cite{deng2020bbig}),
	 if there is an open
immersion $X \subset {\overline X}$ with ${\overline X}$ proper over ${\mathbb C}$ such that, for every morphism $f: \bigtriangleup^*
\rightarrow X^{an}$, there is a morphism 
$\bigtriangleup \rightarrow \overline{X}^{an}$  which extends $f$.
\end{definition}

The classical big Picard theorem can thus be stated as that ${\mathbb P}^1\backslash \{0, 1, \infty\}$ has the
$\bigtriangleup^*$-extension property, or ${\mathbb P}^1\backslash \{0, 1, \infty\}$ is Picard hyperbolic.

Let $X$ be a complex analytic space. Recall that $X$ is called {\it Brody hyperbolic} provided that every holomorphic map $f:\mathbb{C}\to X$ is  constant, and $X$ is said to be {\it Kobayashi hyperbolic} if the Kobayashi pseudo-distance 
$d_X$ is a distance.  It is clear that if $X\setminus D$ is Picard hyperbolic then $X\setminus D$ is Brody hyperbolic.
Conversely,  according to Kwack and Kobayashi,
if $U$  is a quasi-projective variety  and is hyperbolically imbedded in some compactification $\overline{U}$, then $U$  has the
$\bigtriangleup^*$-extension property, i.e. $U$  is Picard hyperbolic.

It is important to indicate that if the $\bigtriangleup^*$-extension property holds, then one
can also obtain the higher-dimensional extension property. 
 This indeed is a consequence of the deep extension theorem of meromorphic maps due to Siu  \cite{siu75}.  The following is the precise statement.

\begin{proposition}[See \cite{deng2020bbig},  Proposition 3.4] 
Let $Y^0$ be a Zariski open set of a compact K\"ahler manifold $Y$. Assume
that $Y^0$ is Picard hyperbolic. Then any holomorphic map $f : \bigtriangleup^p\times 
(\bigtriangleup^*)^q \rightarrow Y^0$
extends to
a meromorphic map $f: \bigtriangleup^{p+q} \rightarrow Y$.   In particular, any holomorphic map $g$ from a Zariski
open set $X^0$ of a compact complex manifold $X^0$ to $Y^0$
extends to  a meromorphic map from $X$ to $Y$.
\end{proposition}

\subsection{Quantitative big Picard and Nevanlinna pair.}

To study the  $\bigtriangleup^*$-extension property, we make a  change of variable $z:=1/\zeta$.
Then the problem is reduced to the statement that there is no essential singularity at $\infty$.
  In general,  corresponding to the big Picard Theorem, one studies the 
 extendability across $\infty$ of a holomorphic map  ${\mathbb C}- \overline{\bigtriangleup(r_0)}\rightarrow X$ to 
  ${\mathbb C}\cup\{\infty\}- \overline{\bigtriangleup(r_0)}\rightarrow X$ for some fixed $r_0\geq 1$, where $X$ is a projective variety.
    In this paper, we use Nevanlinna theory to deal with this problem. 
  Let $\phi: {\mathbb C}- \overline{\bigtriangleup(r_0)}\rightarrow X$ be a holomorphic map.
  Let $\eta$ be a positive  $(1, 1)$-form on $X$. We define, for any fixed $r_1>r_0$,  
$$T_{\phi, \eta}(r, r_1)=\int_{r_1}^r \left(\int_{r_1<|t|< r} f^*\eta\right) \frac{dt}{t}.$$
Sometimes we just write it as $T_{\phi, \eta}(r)$ when $r_1$ is fixed. 
For the $\phi: {\mathbb C}- \overline{\bigtriangleup(r_0)}\rightarrow {\mathbb P}^n({\mathbb C})$, we write, for $r_1 > r_0$,  
 $T_{\phi}(r, r_1): = T_{\phi, \eta}(r, r_1)$ when $\eta$ is taken as the Fubini-Study form of ${\mathbb P}^n({\mathbb C})$.

   The starting point is the following lemma. 
  \begin{lemma}[See \cite{siu2015hyperbolicity}, Proposition 6.2]\label{siu}  
		Let $\phi: {\mathbb C}- \overline{\bigtriangleup(r_0)}\rightarrow {\mathbb P}^n({\mathbb C})$ be a holomorphic map. 
If $T_{\phi}(r, r_1)\leqexc O(\log r)$, then $\phi$  can be extended to a holomorphic
map from ${\mathbb C}\cup\{\infty\}-\overline{\bigtriangleup(r_0)}$ to ${\mathbb P}^n({\mathbb C})$. 
\end{lemma}
  Motivated by the above lemma,  we introduce the following definition.
\begin{definition}\label{quantitativep}
	Let $X$ be a projective variety and $D$ be an effective Cartier divisor on $X$. 	We say a  holomorphic map 
	$f: {\mathbb C}-\overline{\bigtriangleup(r_0)}\rightarrow X\setminus D$ has {\it quantitative big Picard  property} if 
	there exists  a positive $(1, 1)$-form $\eta$ on $X$ such that, 
	for a fixed $r_1$ with $r_1>r_0$, 
	$$T_{f, \eta}(r, r_1)\leqexc  O(\log r).$$
	
	We say that $X\setminus D$ is {\it quantitatively big Picard} if
	 every holomorphic map 
	$f: {\mathbb C}-\overline{\bigtriangleup(r_0)}\rightarrow X\setminus D$ has quantitative big Picard property.
		\end{definition}
From Lemma \ref{siu}, we see that the quantitative big Picard implies Picard hyperbolic. 

  \begin{theorem}\label{qp2} If $X\setminus D$ is hyperbolically imbedded in $X$, then  $X\setminus D$ is quantitatively big Picard.
  \end{theorem}
\begin{proof} Let $f: {\mathbb C}-\overline{\bigtriangleup(r_0)}\rightarrow X\setminus D$ be a holomorphic map.   It suffices to consider the case $r_0 = 1$. Let  $\omega$ be a positive $(1, 1)$-form on $X$ and denote by $\|\cdot\|_{\omega}$ the
	 associated norm. Let $k_{X\setminus D}$ be the Kobayashi-Royden infinitesimal pseudo-norm on $X\setminus D$.
	Since $X\setminus D$ is hyperbolically imbedded in $X$,   there exists a positive real number $c > 0$ such that
		\begin{equation}\label{algarg}
\|\cdot\|_{\omega}	\leq c k_{X\setminus D}. \end{equation}
 On the other hand, by the distance decreasing property of the
 Kobayashi-Royden pseudo-norm,
	$$f^*k_{X\setminus D}\leq k_{ {\mathbb C}- \overline{\bigtriangleup(1)}}.$$
	Therefore $$\|\cdot\|_{f^*\omega}	\leq c k_{ {\mathbb C}- \overline{\bigtriangleup(1)}}.$$ 
	At the level of forms, this yields
	$$f^*\omega\leq c \sqrt{-1} \frac{dz\wedge d\bar{z}}{|z|^2\log^2|z|^2}.$$
	Hence 
	$$\int_{r_1\leq |z|\leq \rho} f^*\omega\leq c\left(\int_{r_1}^{\rho} \frac{1}{t^2\log t} t dt\right)
	= c \left(\frac{1}{\log r_1}-\frac{1}{ \log\rho }\right).$$
	Thus, 
	$$T_{f, \omega}(r, r_1)=  \int_{r_1}^r \left(\int_{r_1\leq |z|\leq \rho} f^*\omega\right) \frac{d\rho}{ \rho}
	\leq C \log r,$$
	where $C$ is a constant depending on $r_1$.
\end{proof}
	
	Note that the above theorem gives a new and simpler proof of the result of 
	Kwack and Kobayashi.
\begin{theorem}[see \cite{deng2020picard}, Theorem A]\label{thm:deng-zuo} 
Let $X$ be a projective manifold and $\omega$ be a K\"ahler
metric on $X$. Let $D$ be a simple normal crossing divisor on $X$. Let $f: \bigtriangleup^*
\rightarrow X\setminus D$ be a
holomorphic map. Assume that there is a Finsler pseudo-metric
$h$ of $T_X(-\log D)$ such that $\|f'(z)\|^2_h\not\equiv 0$, 
$\log \|f'(z)\|^2_h$
is locally
integrable and that the following inequality holds in the sense of currents
$$dd^c [\log \|f'(z)\|^2_h] \geq f^*\omega.$$
Then  $f$ has  quantitative big Picard property.
\end{theorem}
\begin{proof}  We make a  change variable of $z:=1/\zeta$ and consider 
 $f: {\mathbb C}- \overline{\bigtriangleup(1)}\rightarrow X$. 
	From the assumption we get  
	\begin{eqnarray} \label{deng7}
		T_{f, \omega}(r, r_1) &=& \int_{r_1}^r \left(\int_{r_1\leq |z|\leq t} f^*\omega \right){dt\over t} \nonumber \\
		&\leq &   \int_{r_1}^r \left(\int_{r_1\leq |z|\leq t} dd^c [\log \|f'(z)\|^2_h] \right){dt\over t} \nonumber\\
		&=&  \int_0^{2\pi} \log \|f'(re^{i\theta })\|_h \frac{d\theta}{2\pi} -   \int_0^{2\pi} \log \|f'(r_1e^{i\theta })\|_h \frac{d\theta}{2\pi} \nonumber\\
		&= &  \int_0^{2\pi} \log \|f'(re^{i\theta })\|_h \frac{d\theta}{2\pi} +O(1),
		\end{eqnarray}
	where the third equality follows from the Green-Jensen formula.
	We now estimate the integral above. According to (\cite{har77}, I.3), a function $f$ is regular on $X\subset {\mathbb P}^N$ if, for every $P\in X$,  there is a Zariski open neighborhood $U$ 
	with $P\in U$, and homogeneous polynomials $g, h$ in $(N+1)$-variables of the same degree with $h$ is nowhere zero on $U$, and $f=g/h$ on $U$. 
	Since $D$ is a simple normal crossing divisor, we can choose, at each point $P\in \Supp D$,  a local coordinate chart $(U,z_1,\dots,z_n)$ near $P$ 
	such that $D|_U=(z_1\cdots z_s=0)$. From the remark above, we see that  $z_1,\dots,z_n$ are (global) rational functions on $X$.
		Since $X$ is compact, we can cover $X$ by finitely many open subsets (in the complex topology) $\{U_{\lambda}\}_{\lambda\in \Lambda}$ 
		with coordinates $(z_{\lambda, 1},\dots,z_{\lambda, n})$ such that $D|_{U_{\lambda}}=(z_{\lambda,1}\cdots z_{\lambda,s(\lambda)}=0)$ for some $1\leq s(\lambda) \leq n$ (note: $s(\lambda)$ may be empty, i.e.  $D|_{U_{\lambda}}$ may be an empty set). 
	Take  a relative compact subcovering $\{V_{\lambda}\}_{\lambda\in \Lambda}$ with $\overline{V_{\lambda}}\subset U_{\lambda}$ (note that all closure are taken with respect to the complex topology).
Denote by $f_{\lambda, i}=z_{\lambda, i}\circ f$.
	Then, for every $\lambda \in \Lambda$,  there is a constant $C_{\lambda}>0$
such that for $t\in f^{-1}({V_\lambda})$,
	\begin{equation}	
	\begin{split}
	\|f'(t)\|_h^2 
		&\leq  C_{\lambda}\left(
			  \sum_{i=1}^{s(\lambda)} \left| \frac{f_{\lambda, i}'(t)}{f_{\lambda,i}(t)}  \right|^2    +\sum_{j=s(\lambda)+1}^{n}|f_{\lambda, j}'(t)|^2	
			\right) \\
		%% above is first line
		&\leq C_{\lambda}\left(
			\sum_{i=1}^{s(\lambda)} \left|  \frac{f_{\lambda, i}'(t)}{f_{\lambda,i}(t)}   \right|^2
			    + \sum_{j=s(\lambda)+1}^{n}{\hat{C}_{\lambda}|f_{\lambda, j}'(t)|^2\over |f_{\lambda, j}(t)|^2}
			\right),
		%% above is second line
	\end{split}
    \end{equation}
where 	$\tilde{C_{\lambda}}:=\max_{1\leq i\leq n}\sup_{x\in \overline{V_{\lambda}} }z_{\lambda, i}(x)$, which exists because $z_{\lambda, i}$ are holomorphic on $U_\lambda$ and $\overline{V_{\lambda}}$ is compact. Notice that $\Lambda$ is finite, 
	 by the logarithmic derivative lemma,
	\begin{equation} \label{dengn}
	\begin{split}
	\int_0^{2\pi} \log^+ \|f'(re^{i\theta })\|_h \frac{d\theta}{2\pi} 
	&\leq \int_{0}^{2\pi}\log^+  
		\left( \sum_{\lambda\in \Lambda} \sum_{j=1}^n \left|{f_{\lambda j}'\over f_{\lambda j}}(re^{i\theta})\right|\right) 
		{d\theta\over 2\pi}+O(1)\\
	&\leq \sum_{\lambda \in \Lambda}\sum_{j=1}^{n}\int_0^{2\pi} \log^+ \left|{f_{\lambda j}'\over f_{\lambda j}}(re^{i\theta})\right|{d\theta\over 2\pi} + O(1)\\
	&\leq O(\log^+T_{z_{\lambda, j}\circ f}(r, r_1)+\log r)\\
	&\leq O(\log^+T_{f,\omega}(r, r_1)+\log r),
	\end{split}
	\end{equation}
where we used the fact that, for rational function $g$, $T_{g\circ f}(r, r_1)\leq T_{f, \omega}(r, r_1)+O(1)$ (see \cite{noguchi2013nevanlinna}, Theorem 2.13).
So the theorem is proved by  combining  (\ref{deng7}) and  (\ref{dengn}).

\end{proof}

Using the above result,   Deng-Lu-Sun-Zuo \cite{deng2020picard} proved   that {\it
the big Picard theorem holds for the moduli stack $M_h$ of polarized complex
projective manifolds with semi-ample canonical bundle and Hilbert polynomial $h$, i.e., for an algebraic
variety $U$, a compactification $Y$ and a quasi-finite morphism $U \rightarrow M_h$ induced by an algebraic
family over $U$ of such manifolds, $U$ is quantitatively big Picard. }

We introduce the following more general 
definition of {\it Nevanlinna pair}, motivated by the above discussion, as well as the 
 Second Main Theorems in Nevanlinna theory.  We use the Nevanlinna theory on open  parabolic Riemann surfaces recently developed by
Paun and Sibony \cite{paun2014value}.  Note that Stoll \cite{stoll1983ahlfors}  developed the Nevanlinna theory on
a more general parabolic complex manifold through  Ahlfors' approach. Let $Y$ be an open parabolic Riemann surface defined in the introduction. 
Denote by  $B(r):=(\sigma<r)$ the open parabolic disk of radius $r$ and by $S(r):=(\sigma=r)$ the parabolic circle of radius $r$. 
  When $r$ is a regular value of $\sigma$ 
	 (which is the case for almost all $r$), $S(r)$ is
	smooth and one considers on it the measure $d\mu_r: = d^c \log \sigma|_{S(r)}$.  Fix $r_0>0$ such that $\log \sigma$ is harmonic 	outside $\overline{B(r_0)}$. Stokes theorem then implies that, for $r, r_1> r_0$,
	$$\int_{S(r)} d\mu_r - \int_{S(r_1)}  d\mu_r = \int_{B(r) \backslash \overline{B(r_1)}} dd^c  \log \sigma =0.$$
	Thus $\int_{S(r)} d\mu_r$ is independent of $r$ when $r$ is large enough. We denote this constant by $\varsigma$.
	% With this setting we have the generalization of Green-Jensen formula 
	% \begin{proposition}\label{prop:jensensp}
 %    Let $Y$ be a parabolic Riemann surface with parabolic exhausition function $\sigma$. Let $\phi:Y\to [-\infty,\infty)$ be a function on $Y$ which can be written as difference of two subharmonic functions in some local coordinate of every point. Then for $r>1$ large enough
 %    $$\int_{1}^r\frac{dt}{t}\int_{B(t)}dd^c[\phi]=\int_{S(r)}\phi d\mu_r+O(1).$$
 %    \end{proposition}
	Let $X$ be a projective variety and $D$ be an effective Cartier divisor on $X$. 
	 Let 	$f: Y\rightarrow X$ be a holomorphic map, and let $\eta$ be a positive $(1, 1)$ form on $X$.  We define
	$$T_{f, \eta}(r):=\int_{1}^r \left(\int_{B(t)} f^*\eta\right){dt\over t}.$$
	Let $L$ be an ample line bundle on $X$, $h$ be a Hermitian metric on $L$, we write $T_{f, L}(r): =T_{f, c_1(L, h)}(r)$. It is independent, up to a bounded term, the choices of the metric $h$. For an effective Cartier divisor $D$, let $[D]$ denote the corresponding line bundle. Then  $[D]=L_1\otimes L_2^*$ for ample line bundles $L_1,L_2$, and we define $T_{f, D}(r):= T_{f,L_1}(r)-T_{f,L_2}(r)$. Define
		$$m_f(r,D)= \int_{S(r)} \log\frac{1}{\|s_D\circ f\|_h^2}d\mu_r,$$
where $s_D$ is the canonical section of $[D]$ and $h$ is  a Hermitian metric on $[D]$.
For the divisor $f^*D$ on $Y$, we can write $f^*D = \sum_{a\in Y} \nu(a)\cdot a$, where $\nu(a)$ is the order of $f^*D$ at $a$. For an integer $k$, define the  $k$-th truncated divisor $f^*D^{[k]}:=\sum_{a\in Y}\min\{\nu(a),k\}a$.
	 Let $n_{f,D}(r):=\sum_{a\in B(r)}\nu_{f^*D}(a)$ be degree of $f^*D$ counted inside $B(r)$, and $n_{f,D}^{[k]}(r)$ denote its truncated version. In the case $k=1$, $n_{f,D}^{[1]}(r)$ coincides with the number of set theoretic preimage of $D$ in $B(r)$, which we also denote it by $\bar{n}_{f,D}(r)$. Let
	 $$N_{f}(r,D):=\int_{1}^r n_{f,D}(t)\frac{dt}{t},$$
	 and similarly
	 $$N_{f}^{[k]}(r,D):=\int_{1}^{r}n_{f,D}^{[k]}(t)\frac{dt}{t},$$
     $$\overline{N}_f(r, D): = \int_{1}^r \bar{n}_{f, D}(t) {dt\over t}.$$
By identifying the divisor $f^*D$ with its current of integration and by the Poincare-Lelong formula, the following equality holds in the sense of currents:
$$-dd^c[\log \|s_D\circ f\|_h^2] = f^*c_1([D],h) - f^*D.$$
Then the Green-Jensen formula for parabolic Riemann surfaces (see (\ref{eqn:GJ}) below) implies the First Main Theorem
\begin{equation}\label{eqn:np-fmt}
m_f(r,D) + N_f(r,D) = T_{f,D}(r)+O(1).
\end{equation}
 	With these notions, we define the notion {\it Nevanlinna pair} for $(X, D)$ as in Defintion \ref{nevp} in the introduction.
	
	We have shown that if  $(X, D)$ is a Nevanlinna pair then $X\setminus D$ is Brody hyperbolic. Conversely,  similar to 
	 the Theorem 6.2 in \cite{BB}, we have the following result.
\begin{theorem}\label{bb} {\it 	Assume that $X\setminus D$ is hyperbolically imbedded in $X$. 
		Then there is a positive $(1, 1)$-form $\eta$ on $X$ such that for any parabolic Riemann surface $Y$ and every holomorphic map $f: Y\rightarrow X$ with $f(Y)\not\subset D$
		and for $\delta >0$, one has $$T_{f, \eta}(r)\leqexc \overline{N}_f(r, D) - \mathfrak{X}_{\sigma}(r) + (\delta+2\varsigma)\log r + O(1).$$ 
	In particular,  $(X, D)$ is a Nevanlinna pair.}
\end{theorem}
The above theorem, with the term involving $\log r$ explicitly given by $(\delta + 2\varsigma )\log r$, 
   is slightly stronger than the Theorem 6.2 of  Brotbek and  Brunebarbe \cite{BB}.
  Nevertheless we emphasize that we claim no originality
here.

To prove the theorem,  we first recall some notations. Let $M$ be a Riemann surface with a  local coordinate $z=x+\sqrt{-1} y$. 
Denote by 
$$d= \partial +\bar{\partial},~~~~d^c = {\sqrt{-1}\over 4\pi} (\bar{\partial} - \partial)
~~\mbox{so that} ~~dd^c = {\sqrt{-1}\over 2\pi} \partial \bar {\partial}. $$
A conformal  pseudo-metric (resp. metric) on $M$ is given by $h=2 \lambda (dx^2+dy^2)= 2\lambda dzd{\bar z}$ in any local coordinate $(U,z)$ with $z=x+iy$,  where $\lambda$ is a nonnegative (resp. positive) smooth function. It induces a  pseudo-norm on the holomorphic tangent bundle $T_M$ of $M$ defined by $\|\xi\|^2_h=h(\xi, \bar{\xi})= 2\lambda |\gamma|^2$, where $\xi= \gamma {\partial \over \partial z}\in \Gamma(U, T_M)$.  
  Note that, in particular,  
$$\lambda(z) = {1\over 2} \left\|{\partial \over \partial z}\right\|_h^2.$$
The {\it Gaussian curvature} of $h$ is given by
\begin{equation}
K=-{1\over 4\lambda} \bigtriangleup \log \lambda,
\end{equation}
where $\bigtriangleup = {\partial^2 \over \partial x^2} +  {\partial^2 \over \partial y^2}
= 4  {\partial^2 \over  \partial z\partial \bar{z}}$ is the usual Laplacian. 
Let $$\omega = \lambda(z)  {\sqrt{-1}\over 2\pi} dz\wedge d{\bar z}$$
be its {\it associated K\"ahler pseudo form} (or {\it pseudo metric form}).
To $\omega$ we associate the {\it Ricci form} $\mbox{Ric}(\omega): =  dd^c \log \lambda$. Then we have
$$
\mbox{Ric}(\omega) =  dd^c \log \lambda = -K \omega.$$
If $\omega$ is such that in any local coordinate
chart $\log \lambda$ is locally integrable (in which case, we say abusively that $\log \omega$ is locally integrable), one can
define on $M$ the current
$$\operatorname{Ric}[\omega] = dd^c [\log \lambda],$$
which, as in standard, this means that for any smooth function $\phi$ with compact support, one has 
$\operatorname{Ric}[\omega](\phi) := \int_M (\log \lambda)dd^c\phi$. 

We recall the following lemma which is standard in the potential theory. 
\begin{lemma} [See Lemma 2.1 in \cite{BB}]\label{blemma}
	Let $\psi$  be a subharmonic and $C^{\infty}$ function on $\bigtriangleup^*$.
		Suppose that $\psi$  is bounded above. Then $\psi$ 
		extends  to a subharmonic function on $\bigtriangleup$.
		  Moreover, the $(1, 1)$-form $dd^c\psi$ is locally integrable on $\bigtriangleup$ and the
		 following inequality holds in the sense of currents:
		 		$$[dd^c\psi]\leq dd^c[\psi].$$\end{lemma}

\begin{lemma}[Green-Jensen formula, see Proposition 3.1 in \cite{paun2014value}]\label{GJ}	{ Let $Y$ be a non-compact parabolic Riemann
		surface equipped with a parabolic exhaustion function $\sigma$. Let $g: Y\rightarrow  [-\infty, \infty)$ be a function
		such that $dd^c[g]$ is a current with order zero, i.e. $g$ can  locally
		near every point of $Y$  be written as the difference of two subharmonic functions. Then, for  $r\ge 1$ 
		 large enough,
		\begin{equation}\label{eqn:GJ}
		\int_{1}^r {dt\over t} \int_{B(t)} dd^c[g]= \int_{S(r)} g d\mu_r+O(1).
		\end{equation}
		}
\end{lemma}

We need a precise integral formula which relates the exhaustion and Euler characteristic of $B(r)$, which is due to  \cite{paun2014value} with a more detailed proof available in (\cite{BB}, Propositon 2.5). However here we require  the explicit constant involving $\log r$, so we  attach a proof here following the proof of Proposition 2.5 \cite{BB}.
\begin{lemma}[Compare with \cite{paun2014value}, Proposition 3.3]\label{bnew} Same notations as above. we have 
		$$-\int_{1}^r \chi_{\sigma}(t) {dt\over t}  =\int_{S(r)} \log |d\sigma(\xi)|^2 d\mu_r -2\varsigma\log r+O(1).$$
	\end{lemma}	
 
\begin{proof} 
	The proof is similar to the proof of Proposition 2.5 in  \cite{BB}.   
Under our assumption, for the parabolic exhaustion function $\sigma$ of $Y$, $\log \sigma$ is harmonic on $Y$ except at the points $\{P_1, \dots, P_k\}$.  
Let $r>1$ such that $P_i\in B(r)$ for $i=1, \dots, k$. Since $\xi$  is of type (1, 0), one has $d\sigma(\xi) = \partial_{\xi}\sigma$, where $\partial_\xi \sigma$ is the holomorphic directional directive of $\sigma$ by the vector field $\xi$.
Since $\xi$ is holomorphic and $\log \sigma$ is harmonic outside $\{P_1, \dots, P_k\}$,  it follows that $\partial_{\xi}\log \sigma$ is a holomorphic function outside $\{P_1, \dots, P_k\}$. 
A direct computation shows that $\sigma \cdot \partial_{\xi}\log \sigma =\partial_{\xi}\sigma$. Therefore, outside $\{P_1, \dots, P_k\}$ where $\sigma$ has zero, $\partial_{\xi}\sigma$  vanishes only when  $\partial_{\xi}\log \sigma $ does.
Consider the vector field $v = \overline{\partial_{\xi}\sigma}\cdot \xi$. 
Let $(T_Y)_{\mathbb{R}}$ be the real tangent bundle of $Y$. Let $v_{\Bbb R}$ be the corresponding real vector field associated to $v$ via the canonical isomorphism $T_Y\simeq (T_{Y})_{\mathbb{R}}$ given by $\frac{1}{2}(\frac{\partial}{\partial x}-i\frac{\partial }{\partial y})\mapsto \frac{\partial}{\partial x}$, it follows that $v_{\mathbb{R}}$ has singularities on $\{P_1,\dots,P_k\}$ and the zeros of $\partial_{\xi}\log\sigma$. 
Take a holomorphic coordinate $z$ such that $\xi={\partial \over \partial z}$. In this coordinate, one has $\displaystyle{ \overline{\partial_{\xi}\sigma}\cdot \xi =\overline{{\partial \sigma\over \partial z}}{\partial \over \partial z} = {1\over 2}\left({\partial \sigma\over \partial x} + i {\partial \sigma\over \partial y}\right){\partial \over \partial z}}$, so $v_{\mathbb{R}}$ is given by ${\partial \sigma\over \partial x}{\partial \over \partial x} + {\partial \sigma\over \partial y}{\partial \over \partial y}$. 
Observe that $\displaystyle{d\sigma(v_{\mathbb{R}})=\left({\partial \sigma\over \partial x}\right)^2 + \left({\partial \sigma\over \partial y}\right)^2}$, which is strictly positive on the boundary $S(r)$ with our choice of $r$, hence $v_{\mathbb{R}}$  points outwards normal to $B(r)$. 

Since the  real vector field $v_{\Bbb R}$ has only isolated singularities on   $\overline{B}(r)$ and points outwards normal on the boundary of $B(r)$, the Poincar\'e-Hopf theorem implies
	\begin{equation}\label{ru1}
	\chi_{\sigma}(r) =\sum_{p\in B(r)} \operatorname{index}_p(v_{\Bbb R})
	= \sum_{i=1}^k \operatorname{index}_{P_i}(v_{\Bbb R}) + \sum_{p\in B(r)\backslash \{P_1, \dots, P_k\}} \operatorname{index}_p(v_{\Bbb R}). 
	\end{equation}

Now we need to calcualte the indices. We  claim that for every $p \in B(r)\backslash \{P_1, \dots, P_k\}$,  $\operatorname{index}_p(v_{\Bbb R})=-\text{ord}_p (\partial_{\xi}\log \sigma)$. 
To see this, recall that if $g$  is a holomorphic function, then $\displaystyle{\operatorname{index}_0\left(g{\partial\over \partial z}\right)_{\Bbb R}=\text{ord}_0 g}$, where $\left(g{\partial\over \partial z}\right)_{\Bbb R}= \text{Re}(g) {\partial \over \partial x} +\text{Im}(g) {\partial \over \partial y}$ is the real vector field associated to $g {\partial \over \partial z}$. Therefore, in our situation,  $\operatorname{index}_p(\partial_{\xi} \log \sigma\cdot \xi)_{\Bbb R}= 
\text{ord}_p (\partial_{\xi}\log \sigma)$. It follows that for $p \in B(r)\backslash \{P_1, \dots, P_k\}$,
 \begin{eqnarray}\label{ru2}
\text{ord}_p (\partial_{\xi}\log \sigma) &=&	\operatorname{index}_p(\partial_{\xi} \log \sigma\cdot \xi)_{\Bbb R}
	 = \operatorname{index}_p\left({1\over \sigma}\partial_{\xi}  \sigma\cdot \xi\right)_{\Bbb R}\nonumber \\
	&=& \operatorname{index}_p\left(\partial_{\xi}  \sigma\cdot \xi\right)_{\Bbb R} = - \operatorname{index}_p
	\left(\overline{\partial_{\xi}\sigma}\cdot \xi \right)_{\Bbb R} =- \operatorname{index}_p(v_{\Bbb R}).
\end{eqnarray}

This proves our claim. Consequently, for $p\in B(r)\setminus\{P_1,\dots,P_k\}$, $\mathrm{index}_p(v_{\mathbb{R}})$ is equal to the mass of the current $-dd^c[\log|\partial_\xi\log\sigma|^2]$ at $p$ by the Poincar\'e-Lelong formula.

Next we consider the first term on the right hand side of (\ref{ru1}). At each point of $P_i$, $1\leq i\leq k$,  recall our condition on $\sigma$ that $\log \sigma = k_i\log|w| + h_i(w)$ in a local coordinate $w$  centered at $P_i$. Take a better coordidnate $z$ centered at $P_i$ such that $\log \sigma  = k_i \log|z|$, and write $\xi= \psi(z) {\partial \over \partial z}$ for a holomorphic function $\psi$ without zero, then
	$$v = \overline{\partial_{\xi}\sigma}\cdot \xi = \left(\overline{\psi(z) {\partial \over \partial z}(|z|^{k_i})}\right)\psi {\partial \over \partial z}=
	\frac{k_i}{2}|\psi(z)|^2 |z|^{k_i-2}  {z} {\partial \over \partial z}.$$
Hence
\begin{equation}\label{eqn:indexpi} 
		\operatorname{index}_{P_i}(v_{\Bbb R})
	=\operatorname{index}_{0} \left(z\frac{\partial}{\partial z}\right)_{\mathbb{R}}
	= \operatorname{index}_{0} \left(x\frac{\partial}{\partial x}+y\frac{\partial}{\partial y}\right)
	=1. 
\end{equation}
Note that here we used the fact that multiplying a vector field with a real function that is continuous and has no zero in a punctured neighbourhood of the singularity of the vector field does not influence its index at the singularity (see Exercise 7.4, \cite{fulton1995}).
On the other hand, ${\partial \over \partial z} \log \sigma = \frac{\partial}{\partial z}(k_i\log|z|)= \frac{k_i}{2z}$, hence the current $dd^c [\log |\partial_{\xi}\log \sigma|^2]$ has mass $-1$ at $P_i$. Therefore, by combining (\ref{ru1}), (\ref{ru2}), (\ref{eqn:indexpi}), 
	\begin{equation}
	\begin{split}
		\chi_{\sigma}(r) &=k-\int_{B(r) \backslash \{P_1, \dots, P_k\}
		} dd^c [\log |\partial_{\xi}\log \sigma|^2]  \\
	    &= k-\left(\int_{B(r)} dd^c [\log |\partial_{\xi}\log \sigma|^2] +k\right)\\
		&=-\int_{B(r)} dd^c [\log |\partial_{\xi}\log \sigma|^2].
	\end{split}
	\end{equation}
Thus, by the Green-Jensen formula,
	\begin{eqnarray*}
		-\int_1^r \chi_{\sigma}(t) {dt\over t} &=& 
		\int_1^r {dt\over t} \int_{B(t)} dd^c [\log |\partial_{\xi}\log \sigma|^2] \\
		&=& 	\int_{S(r)}  \log |\partial_{\xi}\log \sigma|^2 d\mu_r +O(1)
		= \int_{S(r)} \log  \left|{d\sigma(\xi)\over \sigma}\right|^2d\mu_r +O(1) \\
		&=& \int_{S(r)} \log |d\sigma(\xi)|^2 d\mu_r  - \int_{S(r)} \log |\sigma|^2 d\mu_r +O(1)\\
		&=& \int_{S(r)} \log |d\sigma(\xi)|^2 d\mu_r  - 2\log r \int_{S(r)}  d\mu_r+O(1)\\
		&=& \int_{S(r)} \log |d\sigma(\xi)|^2 d\mu_r -
		2\varsigma \log r +O(1).
	\end{eqnarray*}
This proves our proposition. 
\end{proof}

We also need the following calculus lemma.
\begin{lemma}\label{lem:calculus} {\it 
		Let $H$  be a positive, strictly increasing function defined on $(0, \infty)$. 
		Then the set of $s \in (0, \infty)$ satisfying  the following inequality 
		$$H'(s) > H^{1+\delta}(s)$$
		is of finite Lebesgue measure.}
\end{lemma}

Let $\eta$ be a non-negative $(1,1)$-form on $Y$. 
Note that the relationship between $\eta$ and the pseudo-norm $h_{\eta}$ on  $T_Y$  induced from $\eta$ is given by
\begin{equation}\label{metric} 
\eta = \lambda(z)  {\sqrt{-1}\over 2\pi} dz\wedge d{\bar z}
\longleftrightarrow  h_{\eta}= 2\lambda dzd{\bar z} \longleftrightarrow \|\xi\|^2_{h_{\eta}}=2\lambda |\gamma|^2, 
~\text{if}~\xi= \gamma {\partial \over \partial z}.
\end{equation}
For simplicity, we just write $ \|\cdot\|^2_{{\eta}}$ instead of $ \|\cdot\|^2_{h_{\eta}}$.
Write
 $$T_{\eta}(r) := \int_1^r {dt\over t} \int_{B(t)} \eta.$$
When $Y = \mathbb{C}$, by applying Proposition \ref{lem:calculus} one gets the following consequence of the calculus lemma, 
\begin{equation}
\log \int_{0}^{2\pi}\lambda(re^{i\theta})\frac{d\theta}{2\pi}\leqexc (1+\delta)^2\log T_{\eta}(r)+\delta \log r+ O(1).
\end{equation}
Similarly, 
in the parabolic setting, we have the following corollary of the  calculus lemma.
\begin{lemma}\label{calculus}
Let $\eta$ be a non-negative $(1,1)$-form on $Y$. Then, for any $\delta>0$,
\begin{equation}\label{eqn:applcalculus}
\int_{S(r)} \log \|\xi\|^2_{\eta} ~d\mu_r \leqexc  (1+\delta)^2\log T_{\eta}(r)- \mathfrak{X}_{\sigma}(r) + (\delta+2\varsigma)\log r  +O(1). 
\end{equation}
\end{lemma}
\begin{proof}
We first observe (by the Fubini's theorem) that for any smooth 1-form $\psi$ on $Y$ and any $r>1$,
$$\int_{B(r)} d\sigma\wedge \psi = \int_0^ r \left(\int_{S(t)} \psi\right) dt.$$
Taking derivative on $r$ we get
\begin{equation}\label{eqn:coareaformula}
{d\over dr} \int_{B(r)} d\sigma \wedge \psi = \int_{S(r)} \psi.
\end{equation}
Applying (\ref{eqn:coareaformula}) with $\psi={ \|\xi\|^2_{\eta}\over  |d\sigma(\xi)|^2} d^c\sigma$, we obtain, 
	$$\int_{S(r)} { \|\xi\|^2_{\eta}\over  |d\sigma(\xi)|^2} d\mu_r = {1\over r}{d\over dr} \int_{B(r)} { \|\xi\|^2_{\eta}\over  |d\sigma(\xi)|^2}  d\sigma\wedge d^c\sigma.$$
Now we claim the following identity
	$${\|\xi\|^2_{\eta} \over  |d\sigma(\xi)|^2}  d\sigma\wedge d^c\sigma = 2\eta.$$
Indeed, taking a local coordinate $z$ where $\xi = \frac{\partial}{\partial z}$, by direct computation 
$ d\sigma\wedge d^c\sigma = {\sqrt{-1}\over 2\pi} \left|{\partial \sigma\over \partial z}\right|^2 dz\wedge d\bar{z}$, and  hence 
	$${\|\xi\|^2_{\eta} \over  |d\sigma(\xi)|^2}  d\sigma\wedge d^c\sigma  = { \left\|{\partial \over \partial z}\right\|_{\eta}^2\over   \left|{\partial \sigma\over \partial z}\right|^2}  {\sqrt{-1}\over 2\pi} \left|{\partial \sigma\over \partial z}\right|^2 dz\wedge d\bar{z} = \left\|{\partial \over \partial z}\right\|_{\eta}^2 {\sqrt{-1}\over 2\pi}  dz\wedge d\bar{z}=2\eta,$$ 
which proves the claim. It follows that 
	$$\int_{S(r)} {\|\xi\|^2_{\eta}\over  |d\sigma(\xi)|^2} d\mu_r = {2\over r}{d\over dr} \int_{B(r)} \eta=  {2\over r} {d\over dr} \left( r{dT_{\eta}(r)\over dr}\right).$$
Therefore, by applying the Calculus lemma (Lemma \ref{lem:calculus}) twice, we get, for  $\delta>0$, 
	$$\int_{S(r)} {\|\xi\|^2_{\eta}\over  |d\sigma(\xi)|^2} d\mu_r\leqexc 2 r^{\delta} T^{(1+\delta)^2}_{\eta}(r).$$  
This gives
\begin{equation}\label{eqn:calc-delta}
\log  \int_{S(r)} {\|\xi\|^2_{\eta}\over  |d\sigma(\xi)|^2} d\mu_r\leqexc (1+\delta)^2 \log T_{\eta}(r) + \delta\log r + \log 2.
\end{equation}
Recall from Lemma \ref{bnew} that 
	$$- \mathfrak{X}_{\sigma}(r)= \int_{S(r)} \log |d\sigma(\xi)|^2d\mu_r - 2\varsigma\log r +O(1).$$
Therefore, we get, using  the concavity of the logarithm, 
\begin{equation}
\begin{split}
\int_{S(r)} \log \|\xi\|^2_{\eta} d\mu_r &=  \int_{S(r)} \log \left({\|\xi\|^2_{\eta} \over |d\sigma(\xi)|^2}\right)d\mu_r
+ \int_{S(r)} \log |d\sigma(\xi)|^2 d\mu_r \nonumber\\
&\leq   \log \int_{S(r)}  {\|\xi\|^2_{\eta}\over |d\sigma(\xi)|^2}d\mu_r
 - \mathfrak{X}_{\sigma}(r) + 2\varsigma\log r+O(1) \nonumber\\
&\leqexc  (1+\delta)^2\log T_{\eta}(r) - \mathfrak{X}_{\sigma}(r)+(\delta+2\varsigma)\log r+O(1).
\end{split}
\end{equation} 
This proves the lemma.
\end{proof}

\noindent{\it Proof of Theorem \ref{bb}}. 
The idea  is similar to the proof of Theorem \ref{qp2}. Let $\eta$ be a positive $(1, 1)$-form on $X$.
	Let $\Sigma: = (f^*D)_{red}$ 
	be the set theoretic 
	preimage of the divisor $D$. Let $Y^*: = Y\backslash \Sigma$. Note that $Y^*$ is hyperbolic, and denote by $\omega_{Y^*}$  the 
	 $(1, 1)$-form associated to the 
	Kobayashi metric on $Y^*$.  By the same argument in obtaining (\ref{algarg}), there exists a constant $c>0$ such that 
	$$c f^*\eta \leq \omega_{Y^*}.$$
	This implies, by taking integration, 
	\begin{equation}\label{original}
	c T_{f, \eta}(r) \leq  T_{\omega_{Y^*}}(r):=	\int_1^r {dt\over t} \int_{B(t)}  \omega_{Y^*}.\end{equation}
	We now analyze $\omega_{Y^*}$. Using the fact that $Y^*$ is hyperbolic, we know that the universal cover of $Y^*$ is given by the unit disk $\bigtriangleup$, and 
	the standard (normalized) Poincar\'e metric on unit disk $\bigtriangleup$ descends to the $(1, 1)$-form $\omega_{Y^*}$. 
	Its Gaussian curvature is $K = -1$.
	
	We make the following claims.
	
	(a) Both of the currents $[\operatorname{Ric}\omega_{Y^*}]$ and $\mbox{Ric}[\omega_{Y^*}]$ are well-defined on $Y$.
	
	(b) 	$[\operatorname{Ric}\omega_{Y^*}] \leq  [\Sigma] + \mbox{Ric}[\omega_{Y^*}]$ holds on $Y$.
	
	The statement is local, so we restrict ourselves to a contractible neighborhood $U$ of a point $p\in Y$ (note that $p$ is an isolated point). Let  $z$ be a local coordinate on $U$ centered at $y$, and we assume that $U^*:=U\backslash \{p\}= \bigtriangleup^*$. 
	 	Write  $\omega_{Y^*}=a(z)  \sqrt{-1}dz\wedge d\bar{z}$.	By the distance decreasing property of the Kobayashi metric, we have, for some constant $\delta>0$,	$$a(z) \leq {1\over |z|^2\log^2 (|z|^2\delta)}.$$
	 This implies that $a$ is locally integrable so $\mbox{Ric}[\omega_{Y^*}]$ is well-defined on $Y$.
	Let $\psi(z)=a(z)|z|^2$, then $\log \psi$ is subharmonic in $U^*$ due to the curvature property of $\omega_{Y^*}$,  and is  bounded above due to the inequality above. By Lemma \ref{blemma},  $\psi(z)$ extends to a subharmonic function on $U$ and  
	the $(1, 1)$-form $dd^c\log \psi$ is locally integrable on $U$. This shows that the current $[\operatorname{Ric}\omega_{Y^*}]$ is well-defined on $Y$	since $dd^c\log \psi =dd^c \log a$ outside $p$.  
	Moreover, from Lemma \ref{blemma}, $[dd^c\log \psi]\leq  dd^c[\log \psi]$. Hence, 
		$$[\operatorname{Ric}\omega_{Y^*}]=[dd^c \log \psi]\leq dd^c[\log \psi]=dd^c[\log |z|^2]+ dd^c[\log a]= [\Sigma] + \mbox{Ric}[\omega_{Y^*}].$$
This proves the claim.
	On the other hand, using the condition that the Gauss curvature is $-1$, we have $\omega_{Y^*}=\operatorname{Ric} \omega_{Y^*}$ (in terms of the differential forms). 
	Hence by the claim we get 
	\begin{eqnarray*}
		T_{\omega_{Y^*}}(r) &=& \int_1^r {dt\over t} \int_{B(t)} \omega_{Y^*}
	=	\int_1^r {dt\over t} \int_{B(t)} \operatorname{Ric}\omega_{Y^*} \\
		&\leq&  \int_1^r {dt\over t} \int_{B(t)} [\operatorname{Ric}\omega_{Y^*}] \\
		&\leq&  \overline{N}_f(r, D) + \int_1^r {dt\over t} \int_{B(t)} \mbox{Ric}[\omega_{Y^*}].\end{eqnarray*}
	It remains to estimate the last term. As we have noted above,  $\mbox{Ric}[\omega_{Y^*}]$, regarded as a current on $Y$, is of order zero. Hence the Green-Jensen formula (see Lemma \ref{GJ}) can be applied. 
	 Using the fact that $\operatorname{Ric} [\omega_{Y^*}] =dd^c \log [\|\xi\|^2_{\omega_{Y^*}}]$, 
	 by the Green-Jensen formula (see Lemma \ref{GJ}) and by (\ref{eqn:applcalculus}) we get
	\begin{equation}
	\begin{split}
	\int_1^r {dt\over t} \int_{B(t)} \operatorname{Ric} [\omega_{Y^*}] 
	&= \int_{S(r)} \log \|\xi\|^2_{\omega_{Y^*}}~ d\mu_r +O(1) \\
	&\leqexc    (1+\delta)^2\log  T_{\omega_{Y^*}}(r) -   \mathfrak{X}_{\sigma}(r)+ (\delta+2\varsigma)\log r + O(1).
	\end{split}
	\end{equation}
	Thus 
	$$ T_{\omega_{Y^*}}(r) \leqexc  \overline{N}_f(r, D)+   (1+\delta)^2\log  T_{\omega_{Y^*}}(r) - \mathfrak{X}_{\sigma}(r) + (\delta+2\varsigma)\log r + O(1),$$
	which gives 
	$${1\over 2} T_{\omega_{Y^*}}(r) \leqexc\overline{N}_f(r, D) -   \mathfrak{X}_{\sigma}(r)+ (\delta+2\varsigma)\log r +O(1).$$
	Combing the above with (\ref{original}), we get, for some positive constant $c>0$, 
	$$c T_{f, \eta}(r) \leqexc \overline{N}_f(r, D) -  \mathfrak{X}_{\sigma}(r)+ (\delta+2\varsigma)\log r  +O(1).$$
	This finishes the proof.                $\qed$

\section{Examples of  the Nevanlinna pair $(X, D)$}
We provide examples of  the Nevanlinna pair $(X, D)$ such that
	 $X\setminus D$ may not be  hyperbolically imbedded in $X$. In each case,
	  we obtain a more precise Second Main Theorem 
	type result. To simplify the notation, for a meromorphic function $g$ on $Y$, we use $g'$ to denote the directional derivative $dg(\xi)$. 
	
\subsection{Parabolic version of the logarithmic derivative lemma}
	 The main tool is the logarithmic derivative lemma on parabolic Riemann surfaces due to Paun-Sibony \cite{paun2014value}, slightly modified to serve our purpose. The only change is replacing Proposition 3.3 in their original paper by our Lemma \ref{bnew}, so we omit the detail. 
		
		\begin{proposition}[Compare with \cite{paun2014value}, Theorem 3.7]\label{prop:ldlpaun} Let $f$ be a non-constant meromorphic function on a parabolic Riemann surface $Y$. Then, for $\delta>0$, 
			$$m_{f'/f}(r,\infty)\leqexc \frac{1}{2}
			\left(
			(1+\delta)^2\log^+ T_{f}(r) - \mathfrak{X}_{\sigma}(r)+(\delta+2\varsigma)\log r 
		+ \mathfrak{E}_{\sigma}(r)\right) +O(1),$$
		and
		\begin{eqnarray*}
			m_{\frac{f^{(k)}}{f}}(r,\infty)&\leqexc&  \frac{k}{2}
		\left(
		(1+\delta)^2\log^+ T_{f}(r)- \mathfrak{X}_{\sigma}(r) +(\delta+2\varsigma)\log r 
		+ \mathfrak{E}_{\sigma}(r)\right)\\
		&~&+O(\log^+\log^+T_f(r)+\log\log r).
		\end{eqnarray*}
		\end{proposition}

\subsection{Parabolic version of the logarithmic derivative lemma for jet differentials}

The parabolic version of the logarithmic derivative lemma can be extended to  jet differentials. 
We first recall some notions. Let $X$ be an $n$-dimensional complex manifold. 
We denote by $J_k X:=\bigcup_{x\in X} J_k(X)_x$ the (fiber) bundle of $k$-jets, where $J_k(X)_x$ consists of equivalence classes of germs of  holomorphic curves $f:(\mathbb C,0)\to (X,x)$ with the equivalence relation $f \sim_k g$ if and only if all derivatives $f^{(j)}(0)=g^{(j)}(0)$ coincide for $j=0,\dots,k$. 
The equivalent class of a holomorphic germ $f:(\mathbb C,0)\to (X,x)$ is called the $k$-jet of $f$, denote by $j_kf$.
Note that  $J_k(X)_x$ is isomorphic to ${\mathbb C}^{kn}$ via the identification $f\mapsto (f'(0), \dots, f^{(k)}(0))$.
Observe also that there is a natural ${\mathbb C}^*$-action on the fibers of $J_k X$ given by
$$\lambda\cdot j_kf: =j_k(t\mapsto f(\lambda t)), ~~\forall \lambda\in {\mathbb C}^*.$$
For an open subset $U \subset X$, {\it a jet differential of order $k$ on $U$} is an element $P\in {\mathcal O}(p_k^{-1}(U))$, where $p_k: J_kX\rightarrow X$ is the projection map.
The sheaf of jet differentials is defined to be ${\mathcal E}^{GG}_{k, \bullet}\Omega_X:=(p_k)_*{\mathcal O}_{J_kX}.$ A $k$-jet differential $P \in {\mathcal E}^{GG}_{k, \bullet}\Omega_X(U)={\mathcal O}(p_k^{-1}(U))$ is said to be of weighted degree $m$ if for any $j_kf\in p_k^{-1}(U)$ one has
$$P(\lambda\cdot j_kf)=\lambda^m P(j_kf), ~~\forall \lambda\in {\mathbb C}^*.$$
The {\it Green-Griffiths sheaf} $ {\mathcal E}^{GG}_{k, m}\Omega_X$ is defined to be the subsheaf of ${\mathcal E}^{GG}_{k, \bullet}\Omega_X$
of order $k$ and weighted degree $m$. 
In local coordinates, any element $P\in {\mathcal E}^{GG}_{k, m}\Omega_X (U)$ can be written as
$$P(z, dz, \dots, d^kz)=\sum_{|\alpha|=m} c_{\alpha}(z) (dz)^{\alpha_1}\cdots (d^kz)^{\alpha_k},$$
where $c_{\alpha}\in \mathcal {O}(U)$  for any $\alpha:= (\alpha_1,\dots,\alpha_k) \in ({\mathbb N}^n)^k$, and  where  $|\alpha| := |\alpha_1|+ 2|\alpha_2|+\cdots+k|\alpha_k|$ is the usual multi-index notation for the weighted degree.
The sheaf ${\mathcal E}^{GG}_{k, m}\Omega_X$ is locally free, and we  denote its associated vector bundle by $E^{GG}_{k, m}\Omega_X$.

We briefly recall the construction of Demaily-Semple jet tower as follows, we refer the reader \cite{rosseauimpa} or the original paper of Demailly \cite{Dem} for details. Let ${\mathbb G}_k$ be the group of germs of $k$-jets of biholomorphisms of
$({\mathbb C}, 0)$, that is, the group of germs of biholomorphic maps
$\phi: t \mapsto a_1 t + a_2 t^2+\cdots+a_kt^k$, $a_1\in {\mathbb C}^*$, $a_j\in {\mathbb C}$ for $j>1$,
in which the composition law is taken modulo terms $t^j$
of degree $j>k$. The group $\mathbb{G}_k$ has a natural action on $J_k(X)$ given by $\phi\cdot j_k(f):=j_k(f\circ \phi)$ for any $\phi \in \mathbb{G}_k$.
% The action consists of parametrized $k$-jets of maps $f: ({\mathbb C}, 0) \rightarrow (X, x)$ by a biholomorphic change of parameter $\phi: ({\mathbb C}, 0) \rightarrow ({\mathbb C}, 0)$ defined by $(f,\phi) \mapsto f\circ \phi$. 
Denote by
$$J^{\text{reg}}_k X:=\{j_kf\in J_k X~|~f'(0)\not=0\},$$
the space of non-constant jets.
 % We start with the pair $(X, V)$, where $V \subset T_X$ is a subbundle. First we take $V=T_X$ and  we define $X_1 := {\mathbb P}(T_X)$, and the bundle $V_1 \subset T_{X_1}$ is defined fiberwisely by
% $$
% V_{1,(x,[v])} := \{w \in   T_{X_1}(x,[v]) : d\pi_x(w)\in {\mathbb C}v\}
% $$
% where $\pi: X_1 \rightarrow X$ is the canonical projection and $v\in V$. 
% We also have the following parallel description of $V_1$. 
% Consider a non-constant holomorphic disk $u: ({\mathbb C}, 0) \rightarrow  (X, x)$. 
% We can lift it to $X_1$ and denote
% the resulting germ by $u_1$. Then the derivative of $u_1$ belongs to the $V_1$ directions. 
 Consider the pairs $(X', V')$, where $X'$ is a complex manifold, $V' \subset T_{X'}$ is a subbundle, and denote by $\pi':V'\to X'$ the natural projection. Starting from $(X,T_X)$, define $X_1 := {\mathbb P}(T_X)$, let $\pi_{0,1}:X_1\to X$ the natural projetion.  Define the bundle $V_1 \subset T_{X_1}$  fiberwise by
 $$
 V_{1,(x,[v])} := \{w \in   T_{X_1}(x,[v]) : (d\pi_{0,1})_x(w)\in {\mathbb C}v\}.$$
In other words,  $V_1$ is characterized by the exact sequence
 $$0\rightarrow T_{X_1/X}\rightarrow V_1\xrightarrow{d\pi_{0,1}} {\mathcal O}_{X_1}(-1)\rightarrow 0$$
 where ${\mathcal O}_{X_1}(-1)$  is the tautological bundle on $X_1$, and $T_{X_1/X}$  is the
 relative tangent bundle corresponding to the fibration $\pi_{0, 1}$.
Inductively by this procedure,  
we get the {\it Demailly-Semple tower}
$$
\xymatrixcolsep{2pc}\xymatrix{
	  (X_k,V_k)\ar[r]^-{\pi_{k-1,k}} &(X_{k-1},V_{k-1})\ar[r]^-{\pi_{k-2,k-1}}&\cdots \rightarrow (X_1,V_1)\ar[r]^{\pi_{0,1}}&(X,T_X).
	  }
$$
Then we have an embedding $J^{reg}_k X/{\mathbb G}_k \rightarrow X_k$.
Denote by $X_k^{\text{reg}}$ the image of this embedding in $X_k$ 
and denote by $X_k^{\text{sing}}:=X_k\backslash X_k^{\text{reg}}$. Then $X_k^{\text{sing}}$ is a divisor in $X_k$. 
Denote by  $\pi_k: X_k \rightarrow  X$ the projection, then the direct
image sheaf $\pi_{k*}{\mathcal O}_{X_k}(m)$ is isomorphic to $\pi_{k*}{\mathcal O}_{X_k}(m)\simeq{\mathcal E}_{k,m}\Omega_X$,  whose
sections are precisely the invariant jet differentials $P$, i.e., for any $g\in \mathbb{G}_k$ and any $j_k f\in J^{reg}_k (X)$, 
$P(j_k(f\circ g))=g'(0)^m P(j_k f).$ We shall denote the associated
vector bundle by $E_{k, m}\Omega_X$.  The fiber of $\pi_k$ at a non-singular point of $X$ is denoted by ${\mathcal R}_{n,k}$, 
it is a rational manifold, and it is a compactification of the quotient
$({\mathbb C}^{nk}-\{0\})/{\mathbb G}_k$.

 The above definition can be also extended to the logarithmic setting. 
Let $D$ be a simple normal crossing divisor on $X$ (i.e. $D=D_1+\cdots+D_c$ where $D_1,\dots,D_c$ are smooth irreducible divisors intersecting transversely).
The logarithmic cotangent sheaf, denoted by $\Omega_X(\log D)$, 
 is a locally-free sheaf generated by, on $U$,  $${dz_1\over z_1}, \dots, {dz_{s}\over z_{s}}, dz_{s+1}, \dots, dz_n,$$
 where $U\subset X$ is an open subset with local coordinates $(z_1,...,z_n)$ such that
$D|_U = (z_1\cdots  z_{s} = 0)$.
 Let ${\mathcal J}_k(X,-\log D)$ (called the {\it logarithmic $k$-jet sheaf})  be the sheaf
of germs of local holomorphic sections $\alpha$ of $J_kX$ such that,  for any $\omega\in \Omega_X(\log D)_x$,
$(d^{j-1}\omega)(\alpha)$ are all holomorphic for $j = 1, \dots,k$.
A local meromorphic $k$-jet differential $\omega$ on $U$ is called {\it a logarithmic $k$-jet differential} if $\omega(\alpha)$ is holomorphic for any logarithmic $k$-jet field $\alpha\in  {\mathcal J}_k(X,-\log D)(U)$. The sheaf
of logarithmic $k$-jet differential is denoted by ${\mathcal E}^{GG}_{k,\bullet}\Omega_X(\log D)$. The logarithmic Green-Griffiths sheaf ${\mathcal E}_{k,m}^{GG}\Omega_X(\log D)$ is the
 subsheaf of ${\mathcal E}^{GG}_{k, \bullet}\Omega_X(\log D)$ with
weighted degree $m$. In  local coordinates
$z_1, \dots, z_n$ on $U$ with $D|_U=(z_1\cdots z_s=0)$, any element 
$P\in {\mathcal E}_{k,m}^{GG}\Omega_X(\log D)(U)$ can be written as  
\begin{eqnarray}\label{jdiff}
P(z,dz_1,\dots,dz_n)&=&\sum_{|\alpha|=m} c_{\alpha} \left({dz_1\over z_1}\right)^{\alpha_{1,1}} \cdots \left({d^kz_1\over z_1}\right)^{\alpha_{1,k}}
\cdots \left({dz_s\over z_s}\right)^{\alpha_{s,1}} \cdots \left({d^kz_s\over z_s}\right)^{\alpha_{s,k}}\nonumber\\
&~&\cdot (dz_{s+1})^{\alpha_{s+1, 1}}\cdots (d^kz_{s+1})^{\alpha_{s+1, k}}\cdots (dz_{n})^{\alpha_{n, 1}}\cdots (d^kz_n)^{\alpha_{n, k}},
\end{eqnarray}
where each $c_{\alpha}\in {\mathcal O}(U)$,  the summation is over the $kn$-tuples
$\alpha:= (\alpha_1,...,\alpha_k) \in ({\mathbb N}^n)^k$, and where we used the usual
multi-index notation for the weighted degree $|\alpha| := |\alpha_1|+ 2|\alpha_2|+ \cdots +k|\alpha_k|$. 
 The associated vector bundle is denoted by  $E^{GG}_{k,m}\Omega_X(\log D)$.

We now briefly recall the construction of the logarithmic version of Demaily-Semple tower due to Dethloff-Lu \cite{DL01}. A
 {\it logarithmic directed manifold} is a triple $(X, D, V)$
where $(X,D)$ is a log-manifold, $V$ is a subbundle of $T_X(-\log D)$. 
A morphism between logarithmic directed manifolds $(X',D',V')$ and $(X,D,V)$ is given by a holomorphic map $f:X'\to X$ such that $f^{-1}D\subset D'$ and $f_*V'\subset V$.  
% Given log-manifolds $(X', D')$ and $(X, D)$,
%  a holomorphic map $f: X'\rightarrow X$  such
% that $f^{-1}D \subset D'$ will be called a log-morphism from $(X', D')$ to $(X, D)$.
 The  logarithmic Demailly-Semple $k$-jet tower
 \begin{eqnarray*}
 (X_k(D), D_k, V_k) &\xrightarrow{\pi_{k-1,k}}& (X_{k-1}(D), D_{k-1}, V_{k-1}) \xrightarrow{\pi_{k-2,k-1}}
 \cdots \\
 &\rightarrow&  (X_1(D), D_1, V_1) \xrightarrow{\pi_{0,1}} (X,D, T_X(-\log D))
 \end{eqnarray*}
 is constructed  inductively as follows:  Starting from
 $V_0=V=T_X(-\log D)$, define
$X_k(D):={\mathbb P}(V_{k-1})$, and let $\pi_{k-1, k}: X_k(D)\rightarrow X_{k-1}(D)$
be the natural projection. Set $D_k:=(\pi_{k-1, k})^{-1}(D_{k-1})$
which is a simple
normal crossing divisor. Note that $\pi_{k-1, k}$  induces a morphism
$$(\pi_{k-1, k})_*: T_{X_k(D)}(-\log D_k)\rightarrow (\pi_{k-1, k})^*T_{X_{k-1}(D)}(-\log D_{k-1}).$$
Define $V_k: =(\pi_{k-1, k})_*^{-1}
{\mathcal O}_{X_k(D)}(-1)\subset T_{X_k(D)}(-\log D_k)$, where 
${\mathcal O}_{X_k(D)}(-1): = {\mathcal O}_{{\mathbb P}(V_{k-1})}(-1)$
is the tautological line bundle, which by definition is also a subbundle of
$\pi_{k-1,k}^*V_{k-1}$. 

By Proposition 3.9 in \cite{DL01}, we have
$${\mathcal E}_{k,m}\Omega_X(\log D) =(\pi_k)_*{\mathcal O}_{X_k(D)}(m),$$ 
where  ${\mathcal E}_{k,m}\Omega_X(\log D)$  is the subsheaf of 
${\mathcal E}^{GG}_{k,m}\Omega_X(\log D)$ consisting of
invariant logarithmic differential operator $P$.

Let $f:  Y \rightarrow  X$  be a holomorphic map.
 Let $P\in H^0(X, E_{k, m}^{GG} \Omega_X(\log D))$. Write 
$f^*P: = P(j_kf)$. 
The parabolic version of the logarithmic derivative lemma extends to the following jet differentials (see also \cite{xie},  theorem 3.1).

\begin{theorem}\label{log}  Let $X$ be a complex projective variety  and $D$ be a simple normal crossing divisor on $X$ (possibly empty). 
	 
	 $(a)$ Let $P\in H^0(X, E_{k, m}^{GG} \Omega_X(\log D))$. Then there exists a constant $C>0$ such that for any parabolic Riemann surface $Y$, 
	every holomorphic map 
	$f: Y\rightarrow X$ with $f(Y)\not\subset D$  and with $f^*P\not\equiv 0$,  for $\delta > 0,$ one has 
	$$\int_{S(r)} \log^+| P(j_kf)| d\mu_r \leqexc  C(\log^+ T_{f, E}(r)- \mathfrak{X}_{\sigma}(r) + (\delta+2\varsigma)\log r + \mathfrak{E}_\sigma(r)) +O(1),$$
where $E$ is an (thus any)  ample divisor on $X$.
	
	$(b)$ Let $A$ be an ample line bundle on $X$. For any positive integer $N$, let $P\in 
	H^0(Y, E_{k, m}^{GG}\Omega_X(\log D)\otimes A^{-N})$. Then exists a constant $C>0$ such that 
	 for every holomorphic map $f: Y\rightarrow X$ with $f(Y)\not\subset D$ 
where $Y$ is a parabolic Riemann surface, if 
	$f^*P\not\equiv 0$,  then for $\delta > 0,$
	$$T_{f, A}(r)\leqexc {m\over N} \overline{N}_f(r, D) + 	C(\log^+ T_{f,A}(r) - \mathfrak{X}_{\sigma}(r)+ (\delta+2\varsigma)\log r + \mathfrak{E}_\sigma(r))  +O(1).$$
\end{theorem}
\begin{proof} Same as in the proof of Theorem \ref{thm:deng-zuo}, we can cover $X$ by finitely many open subsets $\{U_{\lambda}\}_{\lambda\in \Lambda}$ with coordinates $(z_{\lambda, 1}, \dots, z_{\lambda, n})$ on $U_{\lambda}$ such that 
	$D|_{U_{\lambda}}=(z_{\lambda,1}\cdots z_{\lambda,s(\lambda)}=0)$. Note that 
 $z_{\lambda, 1}, \dots, z_{\lambda, n}$	are (global)  rational functions on $X$.
	Take  a relative compact subcovering $\{V_{\lambda}\}_{\lambda\in \Lambda}$ with $\overline{V_{\lambda}}\subset U_{\lambda}$ (note that all closure are taking with respect to the complex topology).
	Let $f_{\lambda, i}=z_{\lambda, i}\circ f$. From (\ref{jdiff}), there exists a constant $C_{\lambda}>0$ such that for $z\in Y$ with $f(z)\in V_{\lambda}$, 
\begin{equation}
	\begin{split}
	  \log^+|P(j_k f)(z)| 
	    &\leq C_{\lambda}\sum_{l=1}^k \Bigg(\sum_{i=1}^{s(\lambda)} \log^+\left|{f_{\lambda, i}^{(l)}\over f_{\lambda, i}}(z)\right| +\sum_{t=s(\lambda)+1}^{n} \log^+|f_{\lambda,t}(z)| \Bigg)\\
	    %% ^first line
	    &\leq C_{\lambda} \sum_{l=1}^k \Bigg(\sum_{i=1}^{s(\lambda)} \log^+\left|\frac{f_{\lambda,i}^{(l)}}{f_{\lambda,i}}(z)\right| +\sum_{t=s(\lambda)+1}^n\log^+\left|\frac{f_{\lambda,t}^{(l)}}{f_{\lambda,t}}(z)\right|
	     +\log^+\tilde{C_{\lambda}}\Bigg), 
	     %% ^second line
	\end{split}
\end{equation}
	where  $\tilde{C}_{\lambda}=\max_{s(\lambda) \leq t\leq n}\sup_{x\in \overline{V_\lambda} }z_{\lambda,t}(x)$.
Since $\Lambda$ is finite, 
	 there exist positive constants $C_1,C_2$ such that, on  $Y$,
	\begin{equation}\label{eqn:estimatejet}
	\log^+|P(j_k f)|\leq C_1 \sum_{\lambda\in \Lambda}\sum_{l=1}^k
			\sum_{i=1}^{n }\log^+\left|  \frac{f_{\lambda,i}^{(l)}}{f_{\lambda,i}}  \right| + C_2.
	\end{equation}
For $\delta>0$, the logarithmic derivative lemma (Proposition \ref{prop:ldlpaun}) implies, for each $\lambda\in \Lambda$, $1\leq i\leq n$, there are constants $C_3,C_4 > 0$ such that
	\begin{equation}\label{eqn:loga}
	\begin{split}
	&~\int_{S(r)}\log^+\left|\frac{f_{\lambda,i}^{(l)}}{f_{\lambda,i}}\right|d\mu_r
		\leqexc C_3(\log^+T_{z_{\lambda, i} \circ f}(r)- \mathfrak{X}_{\sigma}(r)+ (\delta+2\varsigma)\log r + \mathfrak{E}_\sigma(r)) +O(1)\\
		&\leqexc C_4(\log^+T_{f,E}(r)-\mathfrak{X}_{\sigma}(r)+ (\delta+2\varsigma)\log r+ \mathfrak{E}_\sigma(r)) +O(1),
	\end{split}
	\end{equation}
	where the last inequality follows from the fact that $\log T_{g\circ f}(r)\leq C(\log T_{f,E}(r))$ for any rational function $g$ on $X$ (see \cite{noguchi2013nevanlinna}, Theorem 2.13). 
	By integrating (\ref{eqn:estimatejet}) on $S(r)$ and applying (\ref{eqn:loga}) we get, for some $C>0$, 
	\begin{equation}
	\int_{S(r)}\log^+|P(j_k f)|d\mu_r \leqexc C(\log^+T_{f,E}(r)- \mathfrak{X}_{\sigma}(r)+ (\delta+2\varsigma)\log r+ \mathfrak{E}_\sigma(r)) +O(1).
	\end{equation}
	This proves (a). 
	
	(b) Since $A$ is an ample divisor on $X$, fix a Hermitian metric $h$ on the line bundle associate to $A$, then the first Chern form $\omega:=c_1(A,h)>0$. 
	% Therefore, without loss of generality, we can 
	% assume that $X={\mathbb P}^N$ and  $h$ (or $\omega_{FS}$) is  the Fubini-Study form of ${\mathbb P}^N$.  
	The Poincaré-Lelong formula implies
	$$ dd^c[\log\|P(j_k f)\|_h^2 ]= Nf^*\omega -[P(j_k f)],$$
	where $[P(j_k f)]$ is the divisor of zero associated to $P(j_k f)$. 
%% This part needs to be rewriten to the standard ldl argument
	Write $D$ locally as $(z_1\cdots z_s=0)$, 
	and let $f_j=z_j(f)$. Notice that  the pole order of $(\log f_j)^{(l)}$ 
	at  $z\in Y$ is at most  $l \min\{\ord_z(f_j), 1\}$. Hence, see that,  from the local expression of $P$ in (\ref{jdiff}), 
	 $$[P(j_k f)] \leq m (f^*D)^{[1]},$$
	 where,  for a divisor $f^*D=\sum_{p\in Y} n_p p$ on $Y$, $(f^*D)^{[1]}:=\sum_{p\in Y} \min\{n_p, 1\} p$.
	Hence
	$$dd^c[\log\|P(j_k f)\|^2_h] \geq N f^*\omega -m\cdot (f^*D)^{[1]}.$$
	Taking integral $\int_{r_1}^r\frac{dt}{t}\int_{B(t)}$ both sides and apply Green-Jensen formula (Lemma \ref{GJ}), we get 
	$$N T_{f, A}(r)\leq m {\overline N}_f(r,D)+\int_{S(r)} \log\|P(j_k f)\|^2_h d\mu_r.$$
	Note that on each $V_\lambda$ in the proof of (a), $\|P(j_k f)\|_h\leq C'_\lambda |P(j_kf)|$ for some $C'_\lambda > 0$. Hence (b) follows from the above inequality and (a).
\end{proof}

\subsection{The complement of hypersurfaces}

Let $X$ be a smooth projective variety and $D$ be a Cartier divisor on $X$. Recall that the 
stable locus of $D$ is 
defined by $B(D):=\bigcap_{m\in {\mathbb N}, F\in |mD|} F$.
As a consequence  of Theorem \ref{log}, we can get the following result.
\begin{corollary}\label{bd1}
 { Let  $X$ be a complex projective variety  and $D$ be a divisor on $X$ with simple normal crossing (possibly empty). Let $A$ be an ample line bundle
on $X$. Let $\pi_{0, k}: X_k(D)\rightarrow X$  be the log Demailly tower associated to the pair $(X, D)$.
For any positive integers $k, N, N'$, assume that the stable base locus 
$$B({\mathcal O}_{X_k(D)}(N)\otimes \pi_{0. k}^*A^{-N'})\subset X_k(D)^{sing}\cup \pi_{0. k}^{-1}(D).$$
 Then exists a constant $C>0$ such that for every holomorphic map $f: Y\rightarrow X$ with $f(Y)\not\subset D$
 where $Y$ is an open parabolic Riemann surface, if 
$f^*{\mathcal P}\not\equiv 0$,  then, for $\delta>0$,
$$T_{f, A}(r)\leqexc {N\over N'}{\overline N}_f(r,D) + 	C(\log^+ T_{f,A}(r)- \mathfrak{X}_{\sigma}(r) +(\delta+2\varsigma)\log r + \mathfrak{E}_\sigma(r))+O(1).$$ }
\end{corollary}

We establish the following result which is an extension of Brotbek and Deng \cite{brotbek2019kobayashi},  Corollary 4.9.

\begin{theorem}\label{thm1}  { Let $X$ be a projective manifold of dimension $n\ge 2$ and let $A$ be
	a very ample line bundle over $X$. Let $D \in |A^m|$ be a general smooth hypersurface with
	$$m\ge (n + 2)^{n+3}(n + 1)^{n+3}.$$ Then there exists a constant $C>0$ such that  for any parabolic Riemann surface $Y$ and every holomorphic map $f: Y\rightarrow X$ with $f(Y)\not\subset D$, for $\delta>0$, one has
	 $$T_{f, A}(r) \leqexc  {\overline N}_f(r,D)+C(\log^+ T_{f, A}(r)- \mathfrak{X}_{\sigma}(r) + (\delta+2\varsigma)\log r  + \mathfrak{E}_\sigma(r))+O(1).$$ 
		In particular $(X, D)$ is a Nevanlinna pair. }
	\end{theorem}

The proof of the above theorem  relies on the following  key result in the paper of Brotbek and Deng \cite{brotbek2019kobayashi}.

\begin{proposition}[\cite{brotbek2019kobayashi}, Corollary 4.5]\label{bd2} {We keep the notation in Corollary 4.5, \cite{brotbek2019kobayashi}: 
				 Let  $\epsilon$ be  a positive integer. Let $k=n+1$, $k'={k(k+1)\over 2}$ and $d=(k+1)n+k$. 
				  Let $l$ be an integer such that $l> d^{k-1} k(\epsilon +kd)$. 
	 Then there exist $\beta, \tilde \beta\in {\mathbb N}$ such that, for any
$\alpha\ge 0$, and for a general hypersurface $D \in |A^{\epsilon+(l+k)d}|$, denoting by $X_k(D)$ the log
Demailly k-jet tower,  the stable base locus
$$B({\mathcal O}_{X_k(D)}(\beta+\alpha d^{k-1}k')\otimes \pi_{0, k}^*A^{\tilde{\beta}+\alpha(d^{k-1} k(\epsilon +kd)-l)})\subset X_k(D)^{sing}\cup \pi_{0, k}^{-1}(D).$$}
\end{proposition}

\noindent{\it Proof of Theorem \ref{thm1}}.
Note that $k = n + 1, d=  (k + 1)n + k = n^2 + 3n + 1$ and set
$$l_0=d^{k-1}k'+d^{k-1}(d+1)^2=d^{k-1}(d+1)\left(d+{3\over 2}\right).$$
By the basic inequality
$$k(k+d-1+kd)<(d+1)^2,$$
one can show that any $m\ge (l_0+k)d+2d$ can be written in the form
$$m=\epsilon+(l+k)d$$ with 
$k\leq \epsilon\leq k+d-1$, and $l\ge d^{k-1}k'+d^{k-1}k(\epsilon +kd)$. 
 In particular, applying
Corollary \ref{bd1} and Proposition \ref{bd2}, we see that for such $m$ and  a general hypersurface $D \in |A^m|$, 
there exists a constant $C>0$ such that
$$T_{f, A}(r)\leqexc {\beta+\alpha d^{k-1}k'\over -\tilde{\beta}-\alpha(d^{k-1} k(\epsilon +kd)-l)}{\overline N}_f(r,D)$$
$$ + 	C(\log^+ T_{f, A}(r)- \mathfrak{X}_{\sigma}(r) + (\delta  +2\varsigma)\log r + \mathfrak{E}_\sigma(r))+O(1).$$ 
However, when $\alpha\rightarrow \infty$, 
$${\beta+\alpha d^{k-1}k'\over -\tilde{\beta}-\alpha(d^{k-1} k(\epsilon +kd)-l)}
\rightarrow {d^{k-1}k'\over l-(d^{k-1} k(\epsilon +kd))}<1.$$
Thus $$T_{f, A}(r) \leqexc {\overline N}_f(r,D)  + 	C(\log^+ T_{f, A}(r) -\mathfrak{X}_{\sigma}(r) + (\delta+2\varsigma)\log r +\mathfrak{E}_\sigma(r))+O(1).$$
It remains to give a bound on 
$(l_0 +
k)d + 2d$. Indeed we have
$$(l_0 +
k)d + 2d =\left(d^{k-1}(d+1)\left(d+{3\over 2}\right)+k+2\right)d$$
$$<(n + 2)^{n+3}(n + 1)^{n+3}.$$
This proves the theorem.

\subsection{The Abelian variety case}

We now consider the abelian variety case, and use Theorem \ref{log} to extend the result of Siu-Yeung \cite{siu1997defects}.
Let $A$ be an abelian variety of dimension $n$. Then the universal covering $\mathbb{C}^{n}\to A$ induces global holomorphic $1$-forms $\omega_1:=dw_1,\dots, \omega_n:=dw_n$, where $(w_1,\dots,w_n)$ are standard coordinates on $\mathbb{C}^n$.  The global $1$-forms give global holomorphic coordinates on $J_k A$ so that $J_kA \simeq A\times {\mathbb C}^{kn}$. We fix such a trivialization for every $k$. The induced coordinates will be called below ``the jet coordinates of $J_k A$".
Let $f:Y\to A$ be a holomorphic map. Fix a global vector field $\xi$ on $Y$, then the lifting of $f$, denoted by 
$j_kf: Y\mapsto A\times J_k A=A\times \mathbb{C}^{kn}$,  is given by 
 $$j_kf=(f, f'_1,\dots, f'_n, \dots, f^{(k)}_1,\dots, f^{(k)}_n)$$
where $f'_i =(f^*\omega_i)(\xi)$ for $i=1, \dots, n$, and inductively, $f_i^{(k)} =(df_i^{(k-1)})(\xi)$.
\begin{theorem}\label{abelian} Let $A$ be an abelian variety. Let $D$ be an ample divisor on $A$. Then there exists a constant $C>0$ such that    for every holomorphic map $f: Y\rightarrow A$ with $f(Y)\not\subset D$
	where $Y$ is a parabolic Riemann surface, we have, for
	some $k_0>0$, 
	$$T_{f, D}(r)\leqexc {\overline N}^{[k_0]}_f(r, D)+C(\log^+ T_{f, D}(r)  - \mathfrak{X}_{\sigma}(r)+  (\delta+2\varsigma)\log r + \mathfrak{E}_\sigma(r))+ O(1).$$
	In particular, 
	$(A, D)$ is 
	a Nevanlinna pair.
\end{theorem}
\begin{proof}  We first deal with the case when $f(Y)$ is Zariski dense in $A$. Let $X_k(f)$ be the Zariski closure of the image of $Y$ under $j_k f$.  Let $I_k$ denote the restriction on $X_k(f)$  of the second projection 
	$\iota_k:J_k A=A\times \mathbb{C}^{nk}\mapsto \mathbb{C}^{nk}$. Let $J_k(D)$ be the $k$-jets of $A$ which are annihilated by $\sigma, d\sigma,\dots,d^k\sigma$  where $\sigma$ is a local defining function of $D$ (when $D$ is a smooth subvariety of $A$ it coincides with the $k$-jet bundle of $D$). 

	Claim: There exists an integer $k \geq 0$ such that $I_k(X_k(f))\cap I_k(J_k(D))\neq I_k(J_k(D))$. 
	
	To prove the claim, fix a point $y_0\in Y$,  it suffices to show $I_k(j_kf(y_0))\notin I_k(J_k(D))$ for some $k\geq 0$.
	 Assume this is not true, i.e. $I_k(j_kf(y_0))\in I_k(J_k(D))$ for all $k$. 
	 Then $J_k(D)\cap I_k^{-1}(I_k(j_kf(y_0)))\neq \emptyset$ for all $k$.
	 	Define $$V_k: = p_k(J_k(D)\cap I_k^{-1}(I_k(j_kf(y_0))))\neq \emptyset,$$
	where $p_{k}$ is the projection $J_k(A)\rightarrow A$. 
	Note that $V_k$ is Zariski closed (because $p_{k}: J_k(A)\rightarrow A$ has a section 
	$\mathrm{id}_A\times \{I_k(j_kf(y_0))\}: A \rightarrow J_k(A)$, and $V_k$ is the pull-back of $\Supp J_k(D)$
	by this section), and note that $V_{k+1}\subset V_k$, 
	 we obtain a decreasing sequence of closed subsets on $A$
	$$D\supset V_1\supset V_2\supset \cdots $$
	which eventually stablizes to a closed set called $V$.  By assumption the elements in the decreasing sequence are non-empty hence $V$ is non-empty. 
	Let $a$ be an element in $V$, so $a+j_kf(y_0)\in J_k(D)$. 
	Define $\tilde{f}(y): =f(y)+a-f(y_0)$, then $\tilde{f}(y_0) = a$, so $j_k\tilde{f}(y_0)\in J_k(D)$ for any $k\geq 0$. By a power series argument,  we get 
 $\tilde{f}(Y)\subset D$, which contradicts to $f$ being algebraically non-degenerate. This proves the claim.

	Note that 
	$I_k$ is proper, therefore $Y_k:=I_k(X_k(f))$ is an irreducible algebraic subset of ${\mathbb C}^{nk}$. 
	By the claim, there is $k_0$ for which there is a polynomial $P$ on ${\mathbb C}^{nk_0}$ satisfying 
	$$P|_{Y_{k_0}}\not\equiv 0, ~~~P|_{J_{k_0}(D)}\equiv 0.$$
	Let $\{U_{\lambda}\}_{\lambda\in \Lambda}$ be a finite open covering of $A$ such that $D\cap U_{\lambda}=(\sigma_{\lambda}=0)$, where $\sigma_{\lambda}$ is regular on a Zariski open neighborhood of $U_{\lambda}$, same as in the proof of Theorem \ref{thm:deng-zuo}, we can view $\sigma_{\lambda}$ as a rational functions on $A$. Then 
	$$\sigma_{\lambda}=d\sigma_{\lambda}=\cdots = d^k\sigma_{\lambda}=0$$
	give defining equations of $J_k(D)|_{U_{\lambda}}$, hence on each $U_{\lambda}$ one obtains the following equation:
	\begin{equation}\label{P}
	a_{\lambda 0}\sigma_{\lambda} + \cdots + a_{\lambda k_0}d^k\sigma_{\lambda}=
	I_{k_0}^*P|_{U_{\lambda}}.
	\end{equation}
	Here $a_{\lambda j}$ are polynomials in jet coordinates with coefficients of rational 
	holomorphic functions on $U_{\lambda}$ restricted on $J_{k_0}(A)|_{U_{\lambda}}$. 
	
	Let $\{h_{\lambda}\}$ be a Hermitian metric on the line bundle $[D]$ associated to $D$, i.e. 
	$$\|\sigma\|^2=h_{\lambda}|\sigma_{\lambda}|^2,$$
	where $\sigma$ is the canonical section of the line bundle $[D]$ (i.e. $[\sigma]=D$).
	Hence from (\ref{P})
	$${|I_{k_0}^*P_{U_{\lambda}}| \over \|\sigma\|}  \leq {1\over h^{1/2}_{\lambda}} 
	\left(|a_{\lambda 0}| + \cdots + 
	|a_{\lambda k_0}|\left|{d^{k_0}\sigma_{\lambda}\over \sigma_{\lambda}}\right|\right).$$
		Choose relatively compact open subsets $V_{\lambda}$ of $U_{\lambda}$ so that $\bigcup_{\lambda} V_{\lambda}=A$. 
Let $w_{ij}, 1\leq i\leq n, 1\leq j \leq k_0$, be the coordinate system on $\mathbb{C}^{nk_0}$ coming from restriction of the jet coordinates of $J_k A$.
Since $a_{\lambda t}$ are polynomials in the jet coordinates with coefficients given by 
holomorphic functions on $U_{\lambda}$, for every $\lambda$, there exist a positive constant $C_{\lambda}$ and an integer $d_{\lambda}>0$ such that,
 for every $z\in Y$ with $f(z)\in V_{\lambda}$ and every   $1\leq t \leq k_0$, 
$$h^{-1/2}_{\lambda}(f(z))|a_{\lambda t} (j_{k_0}f(z))|\leq C_{\lambda} \left(1+ \sum_{1\leq i\leq n, 1\leq j\leq k_0} |w_{ij}(j_{k_0}(f))|\right)^{d_{\lambda}}.$$ 
Hence, for every $z\in Y$ with $f(z)\in V_{\lambda}$, 
\begin{eqnarray*} 
&~&{|(I_{k_0}^*P)(j_{k_0}f(z))|\over \|\sigma(f(z))\|}\\
&\leq& C_{\lambda} 
	 \left(1+ \sum_{1\leq i\leq n, 1\leq j\leq k_0} |w_{ij}(j_{k_0}(f))|\right)^{d_{\lambda}}
	\times	\left (1+\sum_{k=1}^{k_0} \left|{{d^k\over dz^k}
		\sigma_{\lambda}(f(z))\over \sigma_{\lambda}(f(z))}\right|\right).
	\end{eqnarray*}

On the other hand, by Theorem \ref{log}  (part (a)), we get, for any  $1\leq i\leq n, 1\leq j \leq k_0$, 
$$\int_{S(r)} \log^+ |w_{ij}(j_{k_0}f)| d\mu_r \leqexc C_1(\log^+ T_{f, D}(r) - \mathfrak{X}_{\sigma}(r)+(\delta+2\varsigma)\log r+ \mathfrak{E}_\sigma(r))+O(1).$$
Hence
	\begin{eqnarray*} 
		m_f(r, D)&\leqexc&  C_2\left(\sum_{\lambda, 1 \leq j \leq k_0}m\left(r, 
		{(\sigma_{\lambda}\circ f)^{(j)}\over \sigma_{\lambda}\circ f}\right)\right)
		+ m\left(r, {1\over (I_{k_0}^*P)(j_{k_0}f)}\right)\\
		&~&+ C_3(\log^+ T_{f, D}(r) - \mathfrak{X}_{\sigma}(r) + (\delta+2\varsigma)\log r +\mathfrak{E}_\sigma(r))+O(1).
	\end{eqnarray*}
Again, by Theorem \ref{log} (a) or by the logarithmic derivative lemma  
$$m\left(r, 
{(\sigma_{\lambda}\circ f)^{(j)}\over \sigma_{\lambda}\circ f}\right) \leqexc 
C_4(\log^+ T_{f, D}(r) - \mathfrak{X}_{\sigma}(r) + (\delta+2\varsigma)\log r + \mathfrak{E}_\sigma(r))+O(1).$$
Hence, 
\begin{equation}
\begin{split}
m_f(r, D)&\leqexc  m\left(r, {1\over (I_{k_0}^*P)(j_{k_0}f)}\right)\\
 &\qquad+
 C_5(\log^+ T_{f, D}(r)  - \mathfrak{X}_{\sigma}(r)+ (\delta+2\varsigma)\log r+ \mathfrak{E}_\sigma(r))+ O(1).
\end{split}
\end{equation}
 Note here, for a meromorphic function $g$ on $Y$, we also use $m(r, g)$ to denote $m_g(r, \infty)$ and 
 $N(r, g)$ to denote $N_g(r, \infty)$.
By adding $N_{f}(r,D)$ both sides of the above inequality and applying the First Main Theorem (\ref{eqn:np-fmt}), we get
\begin{eqnarray*} 
	T_{f, D}(r)&\leqexc&  
	N_f(r, D)+ m\left(r, {1\over (I_{k_0}^*P)(j_{k_0}f)}\right)\\
	&~&+ C_5(\log^+ T_{f, D}(r) - \mathfrak{X}_{\sigma}(r)+ (\delta+2\varsigma)\log r + \mathfrak{E}_\sigma(r))+O(1).
	\end{eqnarray*}
From (\ref{P}), we see that, for any $z\in Y$, 
$$\ord_z f^*D - \min\{\ord_z f^*D, k_0\}\leq \ord_z ((I_{k_0}^*P)(j_{k_0}f))_0,$$
where $((I_{k_0}^*P)(j_{k_0}f))_0$ is the divisor of zeroes on $Y$ associated to  $(I_{k_0}^*P)(j_{k_0}f)$. 
Hence 
$$N_f(r, D) - N^{[k_0]}_f(r, D)\leq N_{(I_{k_0}^*P)(j_{k_0}f)}(r, 0).$$
Therefore, by the First Main Theorem, 
\begin{equation}\label{eqn:tf}
\begin{split} 
	T_{f, D}(r)&\leqexc  N^{[k_0]}_f(r, D)+  m\left(r, {1\over (I_{k_0}^*P)(j_{k_0}f)}\right) +  N\left(r, {1\over (I_{k_0}^*P)(j_{k_0}f)}\right)\\
	&\qquad+ C_5(\log^+ T_{f, D}(r) - \mathfrak{X}_{\sigma}(r)+  (\delta+2\varsigma)\log r+ \mathfrak{E}_\sigma(r))+O(1)\\
	%% above is first line
	&\leqexc  N^{[k_0]}_f(r, D)+ m(r,   (I_{k_0}^*P)(j_{k_0}f))\\
&\qquad+ C_5(\log^+ T_{f, D}(r) - \mathfrak{X}_{\sigma}(r)+  (\delta+2\varsigma)\log r + \mathfrak{E}_\sigma(r))+O(1).\\	%% above is second line
\end{split}
\end{equation}
Applying Theorem \ref{log} (a) again we get 
\begin{equation}\label{eqn:tik}
m(r,  (I_{k_0}^*P)(j_{k_0}f))\leqexc 
C_6(\log^+ T_{f, D}(r) - \mathfrak{X}_{\sigma}(r) +  (\delta+2\varsigma)\log r + \mathfrak{E}_\sigma(r))+O(1).
\end{equation}
Combining (\ref{eqn:tf}), (\ref{eqn:tik}) finishes the proof in this case.

We now deal with the case when $f(Y)$ is not Zariski dense.  Let $X$ be the Zariski closure of $f(Y)$.
We can assume that $X$ is not a translate of an Abelian sub-variety of $A$, since if it is, then it follows from the condition $f(Y)\not\subset D$ that $m_f(r, D)=m_f(r, D\cap X)$. 
 Hence the above argument can be applied to  the case $A=X$ and $D$ is $D\cap X$ to get our conclusion. Furthermore,  let $A_0$
 be the quotient of the subgroup of all elements whose translates leave $X$
 invariant, i.e., if $B = \{a\in A~|~a + X = X\}$, then $A_0 = A/B$. By replacing
 $f$  by its composite with the quotient map $A\rightarrow A_0$,  we can assume without
 loss of generality that $X$ is not invariant by the translate of any subgroup
 of $A$ with positive dimension. 
 
 Starting
 with $A_0=A$ and $V_0 := T_X$, we consider the  Demaily-Semple jet tower of directed manifolds $(A_k, V_k)_{k\ge 0}$, 
 whose construction was recalled in the begining of subsection 3.2 above.
 Note that $A_k=A \times  {\mathcal R}_{n,k}$ where 
${\mathcal R}_{n,k}$  is the ``universal" rational homogeneous
variety ${\mathbb C}^{nk}/{\mathbb G}_k$. The curve $f: Y\rightarrow A$ lifts to $A_k$, and 
we denote by $f_k: Y\rightarrow A_k$ the lift of $f$.  Let $X_k$ 
 be the Zariski closure of the image of $f_k$, and let
 $\tau_k: X_k\rightarrow {\mathcal R}_{n,k}$
  be the composition of the injection $X_k\rightarrow A_k$  with the projection on
 the second factor  $A_k\rightarrow {\mathcal R}_{n,k}$. According to Proposition 5.3 in \cite{paun2014value}, 
if   for each $k \ge 1$ the fibers of $\tau_k$ are
 positive dimensional, then the dimension of the subgroup $A_X$  of $A$
 defined by
$$ A_X := \{a \in X~|~ a + X = X\}$$
 is strictly positive. On the other hand, we have assumed that 
 $X$ is not invariant by the translate
 of any subgroup of $A$ with positive dimension. 
 Hence we get, for some $k\ge 1$, the  map  $\tau_k: X_k\rightarrow {\mathcal R}_{n,k}$ 
 has finite generic fibers. Thus, by Proposition 5.4 in  \cite{paun2014value}, 
 there exists a jet differential ${\mathcal P}$ of order
 $k$ with values in the dual of an ample line bundle, and whose restriction
 to $X_k$ is non-identically zero. This implies, from Theorem \ref{log}(b) (where the normal crossing divisor is taken to be empty) that, for some $\delta > 0$, 
$$T_{f, D}(r)\leqexc C(\log^+ T_{f, D}(r) - \mathfrak{X}_{\sigma}(r) + (\delta+2\varsigma)\log r + \mathfrak{E}_\sigma(r))+O(1).$$
\end{proof}

\section{The hyperplane case}
In this section we consider the case when $X={\mathbb P}^n({\mathbb C})$ and  $D$ is given by a collection of hyperplanes. 
Let ${\mathcal H}$ be a finite set of hyperplanes in ${\mathbb P}^n({\mathbb C})$.  Let $|{\mathcal H}|:=\sum_{H\in \mathcal{H}}H$. Let $\mathcal{L}$ be the set of defining linear forms of the hyperplanes in $\mathcal{H}$. 
\begin{definition}[\cite{ru95}] 
$\mathcal{H}$ is said to be  {\it non-degenerate} if 
\begin{enumerate}
\item  $\dim({\mathcal L}) =n+1$;  
\item For any proper non-empty subset ${\mathcal L}'$ of ${\mathcal L}$
$$\label{527}
({\mathcal L}')\cap ({\mathcal L}\setminus{\mathcal L}')\cap {\mathcal L}
\not= \emptyset,  $$  
\end{enumerate}
where $({\mathcal L})$ means the vector space generated by the linear forms in ${\mathcal L}$.
\end{definition}
In \cite{ru95} the second named author showed that $\mathbb{P}^n(\mathbb{C})\setminus |\mathcal H|$ is Brody hyperbolic if and only if $\mathcal{H}$ is non-degenerate. Our main purpose of this section is to show that if $\mathcal H$ is non-degenerate then $(\mathbb{P}^n(\mathbb{C}),|\mathcal H|)$ is a Nevanlinna pair. Hence, 
since Nevanlinna pair implies Brody hyperbolic, we get
\begin{theorem}\label{hyper1} 
	$({\mathbb P}^n({\mathbb C}), |{\mathcal H}|)$ is a Nevanlinna pair if and only if ${\mathbb P}^n({\mathbb C})\backslash |{\mathcal H}|$ is Brody hyperbolic. 
\end{theorem}
We first establish a parabolic version of the Ahlfors-Wyel Second Main Theorem.
\begin{theorem}\label{cartan}  Let $Y$ be a parabolic Riemann surface and $f: Y  \rightarrow {\mathbb P}^n({\mathbb C})$ be a linearly  nondegenerate 
	holomorphic map, i.e. its image is not contained in any proper linear subspaces of ${\mathbb P}^n({\mathbb C})$. Let $H_1,\dots,H_q$ be  hyperplanes on $\mathbb{P}^n(\mathbb{C})$ in general position, then, for $\delta',\delta>0$, 
	\begin{eqnarray*}
	&~&	\sum_{j=1}^q m_f(r, H_j)+N_{W}(r, 0)\leqexc
	(n+1+\delta')T_f(r)\\&~& +  \left(\frac{n(n+1)}{2}+\delta'\right)[-\mathfrak{X}_\sigma(r)+(\delta+2\varsigma)\log r]+O(1),
	\end{eqnarray*}
	where $W$ is the Wronskian of $f$.
	\end{theorem}
\begin{proof}
	We modify the geometric proof of the second main theorem for holomorphic curves (see A3.5 \cite{rubook}) to the parabolic setting. For a local coordinate chart $(U,z)$,
    let  ${\bf f}=(f_0,\dots,f_n): U \rightarrow  
	{\mathbb C}^{n+1} - \{0\}$  be a local reduced representation of $f$, where $f_0,\dots,f_n$ are holomorphic functions on $U$ with no common zeroes.
	For $0\leq k\leq 1$, consider the map  ${\bf F}_k$ defined by  
	$${\bf F}_k(z):= 
	{\bf f}(z) \wedge  {\mathbf f'}(z) 
	\wedge  \cdots  \wedge  {\mathbf f^{(k)}(z)}: U \rightarrow 
	\bigwedge^{k+1}{\mathbb C}^{n+1}.$$
	Identify $\bigwedge^{k+1}{\mathbb C}^{n+1}$ with $\mathbb{C}^{N_k+1}$, where $N_k= {(n+1)!\over (k+1)!(n-k)!} - 1$, and let $P:\mathbb{C}^{N_k+1}\to\mathbb{P}^{N_k}(\mathbb{C})$ be the natural projection.
	Then the {\it $k$-th associate map} $F_k:=P(\mathbf{F}_k)$ is a well-defined holomorphic map from $Y$ to $\PP^{N_k}(\mathbb{C})$.

	Let $\omega_k$ be the Fubini-Study form on ${\mathbb P}^{N_k}({\mathbb C})$ and let  $\Omega_k:=F^*_k\omega_k$ be its pull-back on $Y$. Define the $k$-th characteristic function
	$$
	T_{F_k}(r) := \int_1^r {dt\over t} \int_{B(t)}  \Omega_k = T_{\Omega_k}(r).
	$$
Fix $\delta > 0$, let $T(r):=T_{F_0}(r)+\cdots+T_{F_{n-1}}(r)$. We claim that, for $\delta'>0$,
	\begin{equation}\label{claim} 
	T(r) \leqexc (n(n+1)^2+\delta')T_f(r)+ (n(n+1)^2+\delta')[(\delta + 2\varsigma)\log r -\mathfrak{X}_\sigma(r)]+ O(1). 
	\end{equation}		
Now we prove the claim. Following the notation in (\ref{metric}) for $\|\xi\|_{\Omega_k}$, write
$$S_k(r):=  \int_{S(r)}\log \|\xi\|^2_{\Omega_k} d\mu_r.$$
Then (\ref{eqn:applcalculus}) implies
\begin{equation}
\begin{split}
S_k(r)
&\leqexc (1+\delta)^2\log T_{F_k}(r)  +  (\delta+2\varsigma)\log r - \mathfrak{X}_{\sigma}(r) +O(1)\\
&\leqexc (1+\delta)^2\log T(r)+(\delta+2\varsigma)\log r -\mathfrak{X}_\sigma(r)+O(1).
\end{split}
\end{equation}

On the other hand, by the Pl\"ucker formula (see \cite{dhr}, Lemma 4.1 or \cite{rubook}, Lemma A3.5.1),  $\Omega_k = {
		\|{\bf F}_{k-1}(z)\|^2\|{\bf F}_{k+1}(z)\|^2 \over \|{\bf F}_{k}(z)\|^4} \frac{\sqrt{-1}}{2\pi}dz\wedge d\bar{z}$ in any local coordinate $z$. Choose $z$ such that $\xi = \frac{\partial}{\partial z}$, following the notation of (\ref{metric}), we have, for $0 \leq k \leq n$,
	\begin{equation}\label{eqn:pluker}
	\|\xi\|^2_{\Omega_k}= 2{
		\|{\bf F}_{k-1}\|^2\|{\bf F}_{k+1}\|^2 \over \|{\bf F}_{k}\|^4},
	\end{equation}
note that by convention we have set $\|{\bf F}_{-1}\| \equiv 1$. Let $\nu_k$ be the divisor of zeroes of $\Omega_k$.
Applying $\int_{1}^t\frac{dt}{t}dd^c\log[\;\cdot\;]$ on (\ref{eqn:pluker}) and use the Green-Jensen formula (Proposition \ref{GJ}), one has 
	$$
	N_{\nu_k}(r) + T_{F_{k-1}}(r) -2T_{F_{k}}(r)+
	T_{F_{k+1}}(r) =S_k(r).
	$$
From here, by an induction argument (see the proof of Theorem A3.5.3 in \cite{rubook}), for $0 \leq q \leq p$, we can obtain
	$$
	T_{F_p}(r) + (p-q)T_{F_{q-1}}(r) 
	\leq (p-q+1)T_{F_q}(r) + \sum_{j=q}^{p-1} (p-j) S_j(r) +O(1).
	$$
In particular, by taking $q=0$, $p=k$ and notice that $T_{F_{-1}}(r)\equiv 0$, it gives
\begin{equation}
\begin{split}
T_{F_k}(r) &\leq (k+1)T_f(r) +  \sum_{j=0}^{k-1} (k-j) S_j(r) +O(1)\\
&\leq (k+1)T_f(r)+\frac{k(k+1)}{2}[(1+\delta)^2\log T(r)+(\delta+2\varsigma)\log r -\mathfrak{X}_\sigma(r)]+O(1).
\end{split}
\end{equation}
Hence by enlarging the exceptional set in $\leqexc$ if necessary we get for $\delta'>0$
$$T(r)\leqexc (n(n+1)^2+\delta')T_f(r)+(n(n+1)^2+\delta')[(\delta+2\varsigma)\log r -\mathfrak{X}_\sigma(r)]+O(1).$$
The claim is proved.

	Let $0<\epsilon < \frac{\delta'}{n(n+1)^2+\delta'}$ so that by (\ref{claim})
	\begin{equation}\label{eqn:smt-epsilonTr}
	\epsilon T(r)<\delta' T_f(r)+\delta'[(\delta+2\varsigma)\log r -\mathfrak{X}_\sigma(r)] + O(1).
	\end{equation} 
	Let $\mu > 0$ and define
$$
	\lambda:=  {\prod_{k=0}^{n-1}\|{\bf F}_k\|^{2\epsilon}\over \prod_{1\leq j\leq q, 0\leq k\leq n-1}\log^2 (\mu/\phi_k(H_j))},$$
where $\phi_k(H_j)$ are the contact functions (see for example \cite{dhr}, page 14). By a curvature computation (see Theorem 4.6 in \cite{dhr}),   we have
		\begin{equation}\label{eqn:ddch}
		dd^c \log \lambda \ge C \left({\|{\bf F}_0\|^{2(q-(n+1))}\|{\bf F}_n\|^2 \cdot \lambda \over 
			(\|{\bf F}_0\|\cdots \|{\bf F}_{n-1}\|)^{2\epsilon}
			\prod_{j=1}^q |{\bf F}_0(H_j)|^2} \right)^{{2\over n(n+1)}} dd^c|z|^2.
			\end{equation}
In terms of the  notation in (\ref{metric}), 
the above inequality implies that 
	$$\|\xi\|^2_{dd^c \log \lambda} \ge C \left({\|{\bf F}_0\|^{2(q-(n+1))}\|{\bf F}_n\|^2 \cdot \lambda\over 
		(\|{\bf F}_0\|\cdots \|{\bf F}_{n-1}\|)^{2\epsilon}
		\prod_{j=1}^q |{\bf F}_0(H_j)|^2} \right)^{{2\over n(n+1)}},$$
	for some (new) constant $C>0$.
 	Hence,  by the definition and the  Green-Jensen formula (Proposition \ref{GJ}), 
	\begin{equation}\label{eqn:smt-rhs}
	\begin{split}
		{n(n+1)\over 2} \int_{S(r)}  \log \|\xi\|^2_{dd^c \log \lambda}  d\mu_r 
		&\geq \sum_{n=1}^q m_f(r, H_j) - (n+1) T_f(r) +N_W(r, 0)\\
		&\qquad- \epsilon T(r) +   \int_{S(r)}   \log \lambda d\mu_r.
	\end{split}
	\end{equation}
	On the other hand, by (\ref{eqn:applcalculus}) and the Green-Jensen formula (Proposition \ref{GJ}),  
	\begin{equation}\label{eqn:smt-lhs}
	\begin{split}
	\int_{S(r)}  \log \|\xi\|^2_{dd^c \log \lambda}d\mu_r   
	&\leqexc (1+\delta)^2\log  T_{dd^c \log \lambda}(r) +  (\delta+2\varsigma) \log r - \mathfrak{X}_{\sigma}(r)+O(1) \\
	&\leqexc  (1+\delta)^2\log \left(  \int_{S(r)}  \log \lambda d\mu_r \right) + (\delta+2\varsigma)\log r -\mathfrak{X}_{\sigma}(r) +O(1). 
	\end{split}
	\end{equation}
	The proof is finished by observing that
	$$  C_0 \log \left(\int_{S(r)}  \log \lambda d\mu_r \right) -   \int_{S(r)}   \log \lambda d\mu_r$$
	is bounded from above for any $C_0>0$, and combining (\ref{eqn:smt-rhs}), (\ref{eqn:smt-lhs}) with (\ref{eqn:smt-epsilonTr}).
	\end{proof}	

Let $L_1, \dots, L_q$ be the linear forms defining $H_1, \dots, H_q$. For each fixed $y\in Y$, 
	we rearrange the index of $L_1,\dots,L_q$ such that $\ord_{y}(L_1\circ f)\geq \cdots \geq \ord_{y}(L_{q_0}\circ f)\geq n$ ($q_0$ may not exist). Then $\ord_{y}W=\sum_{j=1}^{q_0}(\ord_{y}(f^*H_j)-n)= \sum_{j=1}^q\ord_{y}(f^*H_j)-n)^+$.
Hence 
	$$\sum_{j=1}^qN_f(r,H_j)-N_W(r,0)\leq \sum_{j=1}^qN_f^{[n]}(r,H_j).$$ 
Therefore Theorem \ref{cartan}	and the First Main Theorem imply the following inequality
	\begin{equation}\label{eqn:truncatesmt} 
	\begin{split}
(q-(n+1)-\delta')T_f(r)
&\leqexc \nonumber \sum_{j=1}^qN^{[n]}_f(r,H_j)\\
&\qquad+ \left(\frac{n(n+1)}{2}+\delta'\right)[ -\mathfrak{X}_\sigma(r)+(\delta+2\varsigma)\log r]+O(1).
	\end{split}
	\end{equation}

A linear relation $c_1L_1+\cdots+c_mL_m=0$ is called {\it minimal} provided that all proper subsets of $\{L_1,\dots,L_m\}$ are linearly independent. Note that any linear relation can be reduced to a linear combination of minimal linear relations. We also note that the $L_1, \dots, L_q$, are always pairwisely  linearly independent because the hyperplanes are distinct, hence necessarily $3\le m\le n+2$. Following \cite{ru95}, the idea to prove Theorem \ref{hyper1} is to observe that a minimal linear relation yields a Second-Main-Theorem type inequality (see Lemma \ref{lem:hyp-minimalrelation}) and that, if $\mathcal H$ is non-degenerate, any two hyperplanes in $\mathcal H$ are connected by a chain of minimal linear relations.

Let ${\bf f} = (f_0,\dots,f_n)$ be a local  reduced representation of $f$. Since different reduced representations differ by a holomorphic function with no zero, $\frac{L_i\circ \mathbf{f}}{L_j\circ \mathbf{f}}$ is a well-defined meromorphic function on $Y$ for any $L_i,L_j\in \mathcal L$. 

\begin{lemma}\label{lem:hyp-minimalrelation}
Let $R$ be a minimal linear relation given by $c_1L_1+\cdots+c_uL_{u+1}=0$. Then, for any $\alpha,\beta\in \{1,\dots,u+1\}$ with $\alpha \neq \beta$, 
$$T_{\frac{L_\alpha\circ \mathbf{f}}{L_\beta\circ \mathbf{f}}}(r)\leqexc\sum_{j=1}^q2 N_f^{[n]}(r,H_j)+(n^2+n+2)[-\mathfrak{X}_\sigma(r)+(\delta+2\varsigma)\log r ]+O(1).$$ 
\end{lemma}
\begin{proof}
By rearranging the index if necessary we can assume $\alpha = 1$ and $2\leq \beta \leq u$.
	Let $g_{R}:Y\to \mathbb{P}^{u-1}(\mathbb{C})$ be the holomorphic map defined by 
	$$z\mapsto [c_1(L_1\circ{\bf f})(z),...,c_{u}(L_u\circ{\bf f})(z)].$$
	Let $H'_1:=\{w_1=0\}, \dots, H'_u=\{w_u=0\}, H'_{u+1}=\{c_1w_1+\cdots + c_uw_u=0\}$ be hyperplanes on $\mathbb{P}^{u-1}(\mathbb{C})$.
	Since $R$ is minimal, $g_R$ is linearly non-degenerate. Note that $H_1',\dots,H'_{u+1}$ are in general position. Let $\delta > 0$, by applying Theorem \ref{cartan} (more precisely, (\ref{eqn:truncatesmt})) to $g_{R}$ and $H_1',\dots,H_{u+1}'$ with $\delta'=\frac{1}{2}$, we get
	\begin{equation}\label{3.5}
	\begin{split}
	   T_{\displaystyle{\frac{c_{\beta} L_{\beta}({\bf f})}{c_1
					L_1({\bf f})}}}(r)
		&\leq T_{g_R}(r) \\
		&\leqexc\frac{1}{1-\delta'}\sum_{t=1}^uN_{c_t L_t({\bf f})}^{[u-2]}(r,
		0)+\frac{ \frac{(u-1)u}{2}+\delta' }{1-\delta'} [- \mathfrak{X}_{\sigma}(r)+(\delta+2\varsigma)\log r]+O(1)\\
		&\leqexc2\sum_{t=1}^uN_{L_t({\bf f})}^{[u-2]}(r, 0)+
		(u^2-u+1)[- \mathfrak{X}_{\sigma}(r)+(\delta+2\varsigma)\log r]+O(1)\\
		&\leqexc2 \sum_{t=1}^qN_{L_t({\bf f})}^{[n]}(r, 0)+
		(n^2+n+1)[- \mathfrak{X}_{\sigma}(r)+(\delta+2\varsigma)\log r]+O(1),
	\end{split}
	\end{equation}
	where the last equality holds because $u\leq n+1$.
\end{proof}

\begin{lemma}\label{lem:nondeg}
	If $\mathcal{H}:=\{H_1,\dots,H_q\}$ is non-degenerate, then there exist $n+1$ linearly independent linear forms
	$L_{i_1},\dots, L_{i_{n+1}}$ in
	$\mathcal L$ such that, for $\delta > 0$,
		\begin{equation}\label{3.3}
			\begin{split}
		T_{\displaystyle{\frac{L_{i_\alpha}({\bf f})}
			{L_{i_1}({\bf f})}}}(r)  &\le_{exc}  (n-1) \left(2\sum_{j=1}^q
		N_f^{[n]}(r, H_j)
		+ (n^2+n+2)[- \mathfrak{X}_{\sigma}(r)+(\delta+2\varsigma)\log r]\right)\\
&~		 +O(1),
\end{split}
		\end{equation} 
	for $2\le \alpha\le n+1$.
	\end{lemma}
	\begin{proof}
	Choose an arbitrary $L_1\in \mathcal{L}$ and let $R_1$ be a minimal linear relation containing $L_1$. Let $\mathcal{L}_1$ be the set of linear forms appeared in $R_1$. Then for any $L_{i_1}\in \mathcal L_1$, Lemma \ref{lem:hyp-minimalrelation} implies
	$$T_{\displaystyle\frac{L_{i_1}\circ \mathbf{f} }{L_1\circ \mathbf{f} } } 
	\leqexc \sum_{j=1}^q2 N_f^{[n]}(r,H_j)+(n^2+n+2)[-\mathfrak{X}_\sigma(r)+(\delta+2\varsigma)\log r ]+O(1).
	$$
	If $\# \mathcal{L}_1=n+2$, we are done. Otherwise let $\mathcal{L}_2:=\mathcal L\cap (\mathcal L_1)$. Since $\mathcal{H}$ is non-degenerate, at least one element of $\mathcal{L}_2$ belong to $(\mathcal L\setminus \mathcal L_1)$. Then for any element $L_{i_2}$ in $\mathcal L_2$, either $L_{i_2}$ is in $\mathcal L_1$, or $L_{i_2}$ satisfies a minimal linear relation involving an element $L_{i_1}$ in $\mathcal L_1$. It sufficies to consider the latter case. By Lemma \ref{lem:hyp-minimalrelation},
	$$T_{\displaystyle\frac{L_{i_2}\circ \mathbf{f} }{L_{i_1}\circ \mathbf{f} } } 
	\leqexc \sum_{j=1}^q2 N_f^{[n]}(r,H_j)+(n^2+n+2)[-\mathfrak{X}_\sigma(r)+(\delta+2\varsigma)\log r ]+O(1).
	$$
	Observe that $\frac{L_{i_2\circ \mathbf{f} }}{L_1\circ \mathbf{f}}
	=\frac{L_{i_2\circ \mathbf{f} }}{L_{i_1}\circ \mathbf{f}} \cdot
	\frac{L_{i_1\circ \mathbf{f} }}{L_1\circ \mathbf{f}}$, hence
	$$T_{\displaystyle \frac{L_{i_2}\circ \mathbf{f} }{L_{1}\circ \mathbf{f} } } 
	\leqexc 2\left( \sum_{j=1}^q2N_f^{[n]}(r,H_j)+(n^2+n+2)[-\mathfrak{X}_\sigma(r)+(\delta+2\varsigma)\log r ] \right)+O(1).
	$$
	If $\#\mathcal{L}_2=n+2$, then we are done. Otherwise, note that $\mathcal{L}_2\supsetneq \mathcal{L}_1$. Inductively, since $\dim(\mathcal L)=n+1$, we obtain a finite sequence $\mathcal{L}_s\supsetneq \mathcal{L}_{s-1}\supsetneq \cdots \supsetneq \mathcal{L}_1$ with $\#\mathcal L_{s}=n+2$, and for any $L_{i_s}\in \mathcal{L}_s$, 
	$$T_{\displaystyle \frac{L_{i_s}\circ \mathbf{f} }{L_{1}\circ \mathbf{f} } } 
	\leqexc s\left( \sum_{j=1}^q2 N_f^{[n]}(r,H_j)+(n^2+n+2)[-\mathfrak{X}_\sigma(r)+(\delta+2\varsigma)\log r ] \right)+O(1).
	$$
	The proof is finished by noting that $\#\mathcal L_1\geq 3$ and hence $s\leq n-1$. 
	\end{proof}

\noindent{\it Proof of Theorem \ref{hyper1}}.

	By Lemma \ref{lem:nondeg}, there are $n+1$ linearly independent linear forms $L_{i_1}, \dots, L_{i_{n+1}} \in {\mathcal L}$ such that (\ref{3.3}) holds. Hence 
	\begin{equation}\label{3.11}
	\begin{split}
&~	T_f(r)
	\leq \sum_{j=2}^{n+1} T_{\displaystyle{\frac{L_{i_j}({\bf f})}{L_{i_1}({\bf f})}}}(r)  \\
	&\leqexc n(n-1) \left(2\sum_{j=1}^q
		N_f^{[n]}(r, H_j)
		+ (n^2+n+2)[(\delta+2\varsigma)\log r - \mathfrak{X}_{\sigma}(r)]\right) +O(1)\\
	&\leqexc2n^2(n-1)\overline{N}_{f}(r,|\mathcal H|)+n(n-1)(n^2+n+2)[
		(\delta+2\varsigma)\log r -\mathfrak{X}_\sigma(r)
		]+O(1).
	\end{split}
	\end{equation}
	This shows that if ${\mathcal H}$ is non-degenerate, then $(\mathbb{P}^n(\mathbb{C}),|\mathcal H|)$ is a Nevanlinna pair.

\section{Nevanlinna pair and Algebraic hyperbolicity}

The concept of algebraic hyperbolicity for a compact complex manifold $X$ was
introduced by Demailly in \cite{Dem}, Definition 2.2, and he proved (see \cite{Dem},  Theorem 2.1) that
$X$ is algebraically hyperbolic if it is Kobayashi hyperbolic. The notion of algebraic hyperbolicity
was generalized to the case of log-pairs $(X, D)$ by Chen \cite{chen}.
According to Chen,  $(X, D)$ is said to be  {\it algebraically hyperbolic}
if there exists a positive $(1,1)$-form $\omega$ on $X$ such that 
 for any compact Riemann surface $R$ and every holomorphic map 
$f: R\rightarrow X$ with $f(R)\not\subset D$, one has
$$\int_R f^*\omega \leq \bar{n}_f(D)+\max\{0, 2g-2\},$$ 
where $\bar{n}_f(D)$ is the number of points of $f^{-1}(D)$ on $R$ and $g$ is the genus of $R$.

Unlike Demailly's theorem (as well as the theorem of Pacienza-Rousseau \cite{pacienza2007logarithmic} for log-pairs $(X, D)$), it is unclear whether  Kobayashi hyperbolicity
or Picard hyperbolicity of $X\setminus D$ will imply the algebraic hyperbolicity of $(X,D)$. 
 However, we prove that $(X,D)$ is algebraically hyperbolic if $(X,D)$ is a Nevanlinna pair, which is one of the main points and motivations of this paper.
 \begin{theorem}\label{alg} { If $(X, D)$ is a Nevanlinna pair, then $(X,D)$ is  algebraically hyperbolic.}
\end{theorem}
\begin{proof}
Let $R$ be a compact Riemann  surface with genus $g$ and $f: R\rightarrow X$ be holomorphic map with $f(R)\not\subset D$. We need to show that 
\begin{equation}\label{eqn:alghyp}
\int_R f^*\omega \leq \bar{n}_f(D)+\max\{0, 2g-2\}
\end{equation}
for a positive $(1,1)$-form $\omega$ on $X$ that is independent of $R$ and $f$.

Fix a point $Q\in R$ such that $f(Q)\not\in \Supp(D)$. 
We view $Q$ as a divisor on $R$ with degree $1$. 
Let $L(kQ)$ be the vector space of meromorphic functions $\psi$ on R such that either $\psi$ is a constant or $(\psi) + kQ \geq 0$, i.e.,  $\psi$ has only a pole at $Q$ with order less than or equal to $k$.  
By the Riemann-Roch Theorem$$\dim L(kQ)- \dim L(kQ-K) = k - g + 1,$$ where $K$ is the canonical divisor on $R$. 
This implies that $\dim L((g+1)Q)\ge 2$,
so we can choose a non-constant meromorphic function $\psi$ on $R$ with a single pole at $Q$ of order less than or equal to $g + 1$.  Then $\sigma:=|\psi|$ is a parabolic exhaustion function for the open parabolic Riemann surface $R\backslash \{Q\}$. 
By the Poincar\'e-Lelong and Stokes' formula, for $r>1$ such that all zeros of $\psi$ are inside  $B(r)$, 
$$g+1\geq \sum_{p\in B(r)} \text{ord}_p \psi = \int_{B(r)} dd^c [\log |\psi|^2] = 2\int_{S(r)} d^c \log |\psi| = 2\int_{S(r)} d^c\log\sigma =2\varsigma.$$
Hence $\varsigma \leq{g+1\over 2}$. 
Since $(X, D)$ is a  Nevanlinna pair, there exists a positive $(1,1)$-form $\eta$ on $X$ such that 
\begin{equation}\label{algrunew}
T_{f, \eta}(r)\leqexc \overline{N}_f(r, D)- \mathfrak{X}_{\sigma}(r) + (\delta+g+1) \log r + \mathfrak{E}_{\sigma}(r)+O(1),\end{equation}
where we used the fact $\varsigma \leq{g+1\over 2}$. 
Also, since $\psi$ has a pole only at $Q$, $\psi':= d\psi(\xi)$ also has a pole only at $Q$ (otherwise $\psi'$ would be a constant). 
Hence $\mathfrak{E}_{\sigma}(r): = \int_{S(r)}\log^- |d\sigma(\xi)|^2 d\mu_r=0$ for $r$ big enough. 	
Therefore, from (\ref{algrunew}) we can	take a sequence $r_n\rightarrow +\infty$ such that $$T_{f, \eta}(r_n)\leq \overline{N}_f(r_n, D) - \mathfrak{X}_{\sigma}(r_n)+ (\delta+g+1)\log r_n+O(1).$$ 
Now recall that 
$$\mathfrak{X}_{\sigma}(r)=\int_{1}^r \chi_{\sigma}(t){dt\over t},$$ where $\chi_{\sigma}(t)$ is the Euler characteristic of the domain $B(t)$.
Hence, 
$$\lim_{r\rightarrow \infty} {\mathfrak{X}_{\sigma}(r)\over \log r}=\chi(R-\{p\})=\chi(R)-1=1-2g.$$
Let 
	$$A(r):= \int_{B(r)} f^*\eta.$$ 
	Then, for any fixed $r$, when $r_n>r$
\begin{eqnarray*}
 	A(r) &\leq& {1\over \log r_n-\log r} \int_r^{r_n} A(t) {dt\over t} 
 	\leq {1\over \log r_n-\log r} T_{f, \eta}(r_n)\\
 	&\leq& {1\over \log r_n-\log r}  (\overline{N}_f(r_n, D)- \mathfrak{X}_{\sigma}(r_n) + (\delta+g+1)\log r_n +O(1))\\
 	&\leq & {1\over \log r_n-\log r} (\bar{n}_f(D) \log r_n - \mathfrak{X}_{\sigma}(r_n)+ (\delta+g+1)\log r_n +O(1)).
 	\end{eqnarray*}
By taking $n\rightarrow \infty$ we get 
$$A(r) \leq  \bar{n}_f(D) +\delta+3g.$$
Now let $r\rightarrow +\infty$ and then let $\delta \to 0$, one gets
	$$\int_R f^*\eta \leq \bar{n}_f(D)+3g.$$ 
Observe that when $g\geq 2, n_f(D)\geq 0$ or $0\leq g\leq 1$ with $\bar{n}_f(D)\geq 1$, we have $\bar{n}_f(D)+3g\leq 6(\bar{n}_f(D)+\max\{0,2g-2\})$. Hence, by choosing $\omega={1\over 6}\eta$,  
$$
\int_R f^*\omega \leq \bar{n}_f(D)+\max\{0, 2g-2\},$$
which verifies (\ref{eqn:alghyp}).
It remains to deal with the following two cases, 
\begin{enumerate}
	\item $g=0, \bar{n}_f(D)=0$,
	\item $g=1, \bar{n}_f(D)=0$.
\end{enumerate}
In these two cases, we prove that  $f$ must be a constant so (\ref{eqn:alghyp}) trivially holds. In case (1), $R=\mathbb{P}^1(\mathbb{C})$, so $f$ is a holomorphic map from $\mathbb{P}^1(\mathbb{C})$ to $X\setminus D$, which must be constant because $X\setminus D$ is Brody hyperbolic. In case (2), consider the universal covering $\pi:\mathbb{C}\to R$, let $\tilde f:=f\circ \pi$ be the lifting of $f$. Then $\tilde{f}:\mathbb{C}\to X\setminus D$ is a holomorphic curve. Since $X\setminus D$ is Brody hyperbolic, $\tilde{f}$ is constant, so $f$ must be constant.
\end{proof}

We have the following consequences of Theorem \ref{alg}.

\begin{corollary}
\begin{enumerate}
	\item If  { ${\mathbb P}^n({\mathbb C})\setminus |\mathcal H|$  
is Brody hyperbolic, then $({\mathbb P}^n({\mathbb C}), |\mathcal H|)$ is algebraically hyperbolic.}
	\item { If $A$ is an abelian variety and $D$ is an ample divisor, then $(A, D)$ is algebraically hyperbolic.}
	\item { Let $X$ be a projective manifold of dimension $n\ge 2$ and let $A$ be a very ample line bundle over $X$. If $D \in |A^m|$ is a general smooth hypersurface with
$$
m\ge (n + 2)^{n+3}(n + 1)^{n+3},
$$
then $(X,D)$ is algebraically hyperbolic. }
\end{enumerate}
\end{corollary}
\begin{proof}
(1) By combining Theorem \ref{hyper1} with Theorem \ref{alg}. 

(2) By combining Theorem \ref{abelian} with Theorem \ref{alg}. 

(3) By combining Theorem \ref{thm1} with Theorem \ref{alg}. 
\end{proof}

\bigskip
\noindent{\bf Acknowledgment.}
	We thank Professors Nessim Sibony and Xiaojun Huang for many helpful discussions.

%\bibliographystyle{abbrv}
%\bibliography{../bib-data-abchyp}	

\end{document}